\documentclass[11pt,english]{article}
\usepackage{babel}
\usepackage[cp1250]{inputenc}
\usepackage[T1]{fontenc}
\usepackage{amsfonts}
\usepackage{amsmath}
\usepackage{amsthm}
\usepackage{mathrsfs}
\usepackage{hyperref}
\usepackage{enumerate}
\usepackage{graphicx}
\usepackage{pictex}

\newcommand*{\norm}[1]{\left\Vert{#1}\right\Vert}
\newcommand*{\abs}[1]{\left\vert{#1}\right\vert}

\newcommand*{\Om}{\Omega}
\newcommand*{\mr}{\mathbb{R}}
\newcommand*{\izi}{\int_{0}^{\infty}}
\newcommand*{\izt}{\int_{0}^{t}}
\newcommand*{\izj}{\int_{0}^{1}}
\newcommand*{\iO}{\int_{\Omega}}
\newcommand*{\p}{\partial}
\newcommand*{\poch}{\frac{d}{dt}}
\newcommand*{\rlm}{\p^{(\mu)}}
\newcommand*{\mb}{\mu(\beta)}

\newcommand*{\dm}{D^{(\mu)}}
\newcommand*{\da}{D^{\alpha}}
\newcommand*{\im}{I^{(\mu)}}
\newcommand*{\al}{\alpha}
\newcommand*{\uz}{u_{0}}

\newcommand*{\tl}[1]{\widehat{#1}}

\newcommand{\ald}{d \al}

\newcommand*{\ma}{\mu(\alpha)}
\newcommand*{\mg}{\frac{\mu(\alpha)}{\Gamma(1-\alpha)}}
\newcommand*{\ggg}{\Gamma(\gamma)}
\newcommand*{\gjg}{\Gamma(1-\gamma)}

\newcommand*{\vf}{\varphi}
\newcommand*{\ve}{\varepsilon}

\def\Re{\operatorname {Re}}
\def\Im{\operatorname {Im}}
\def\arg{\operatorname{arg}}
\def\divv{\operatorname {div}}

\newcommand{\eqq}[2]{\begin{equation}  #1  \label{#2}\end{equation}    }
\newcommand{\hd}{\hspace{0.2cm}}
\newcommand{\hj}{\hspace{0.1cm}}
\newcommand{\no}{\noindent}
\newcommand{\m}[1]{\mbox{#1}}

\newcommand{\OT}{\Omega^{T}}

\newtheorem{rem}{{\textbf {Remark}}}
\newtheorem{lem}{{\textbf {Lemma}}}
\newtheorem{prop}{{\textbf {Proposition}}}
\newtheorem{theorem}{\textbf {Theorem}}
\newtheorem{coro}{\textbf  {Corollary} }
\newtheorem{de}{\textbf  {Definition} }

\newcommand{\ia}{I^{\alpha}}
\newcommand{\ija}{I^{1-\alpha}}

\newcommand{\ep}{\varepsilon}
\newcommand{\tamj}{(t-\tau)^{\alpha-1}}
\newcommand{\dt}{d\tau}

\newcommand{\ga}{\Gamma(\alpha)}
\newcommand{\gja}{\Gamma(1-\alpha)}
\newcommand{\jga}{\frac{1}{\ga}}
\newcommand{\jgja}{\frac{1}{\gja}}

\newcommand{\lap}{\Delta}

\newcommand{\vk}{\varphi_{k}}

\newcommand{\vp}{\varphi}

\newcommand{\ra}{\partial^{\alpha}}

\newcommand{\un}{u^{n}}

\newcommand{\io}{\int_{\Omega}}

\newcommand{\ta}{(t-\tau)^{-\alpha}}

\newcommand{\ld}{L^{2}(\Omega)}

\newcommand{\hjzo}{H^{1}_{0}(\Omega)}

\newcommand{\aij}{a_{i,j}}

\newcommand{\fjn}{f_{\frac{1}{n}}}

\newcommand{\ddt}{\frac{d}{dt}}

\newcommand{\hk}{\bar{H}}
\newcommand{\hkk}{\hk^{k}}

\newcommand{\hkks}{(\hk^{k})^{\ast}}

\newcommand{\cmt}{c_{\mu, T}}

\newcommand{\gs}{\gamma_{*}}

\newcommand{\tgj}{t^{\gamma-1}}
\newcommand{\tjg}{t^{1-\gamma}}

\newcommand{\map}{\widetilde{\mu}(\al)}

\begin{document}
\title{\bf Fractional diffusion equation with the distributed order  Caputo derivative}

\author{ Adam Kubica, Katarzyna Ryszewska\footnote{Department of Mathematics and Information
Sciences, Warsaw University of Technology, ul. Koszykowa 75, 00-662 Warsaw,
Poland, E-mail addresses:
A.Kubica@mini.pw.edu.pl, K.Ryszewska@mini.pw.edu.pl}}

\maketitle

\abstract{We consider fractional diffusion equation with the distributed order Caputo derivative. We prove existence of a weak and regular solution for general uniformly elliptic operator under the assumption that the weight function is only integrable.}

\vspace{0.3cm}

\no Keywords: distributed-order fractional diffusion, weak solutions, continuity at initial time.

\vspace{0.2cm}


\section{Introduction}

In this paper we consider parabolic type equation with the distributed order time fractional Caputo derivative and general elliptic operator with time-depended coefficients. We prove existence of a unique weak and regular solution. In our case the distributed order fractional derivative is defined as a weighted fractional Caputo derivative, where the weight $\mu$ is supposed to be any nonnegative nontrivial function from $L^{1}(0,1)$. This kind of problems were studied in many papers (see \cite{Al}, \cite{cmf}, \cite{Kochubei}-\cite{Kochubei3}, \cite{Li}, \cite{LiLuchkoY}, \cite{LuchkoY}, \cite{Rundell} ), however the authors  usually impose stronger assumption on $\mu$ and consider time-independent elliptic operators. These assumptions make the analysis easier, because one can apply the  Laplace transform and solution can be defined by means of Fourier series.

Our approach is rather general and is based on Galerkin method and energy estimates obtained for a special approximating sequence. We reconstruct the reasoning from \cite{Zacher}, however we do not need Yosida approximation to deal with the Caputo derivative. This is the main advantage of our treatment, which is elementary and can be applied to the more complicated problems (see \cite{Stefan}).

Furthermore, we investigate the correctness of weak form of the Caputo derivative proposed by Zacher  (\cite{Zacher}) and examine the continuity of solution at initial time. Our result (see theorem~\ref{tw2}) is related only to problems with Laplace operator and the case of general elliptic operator will be analyzed in another paper.

We recall the definition of fractional integration operator  $\ia $ and
the Riemann-Liouville \ fractional  derivative
\eqq{\ia f (t) = \jga \izt \tamj  f(\tau) \dt \hd  \m{ for }\alpha>0,}{fI}
\eqq{\hspace{-0.1cm} \partial^{\alpha}f(t) \hspace{-0.1cm}  =\hspace{-0.1cm} \ddt \ija f(t) \hspace{-0.1cm} = \hspace{-0.1cm} \jgja \ddt \hspace{-0.1cm} \izt \hspace{-0.1cm} \hspace{-0.1cm}  \ta f(\tau)
\dt \m{ for }\alpha\in [0,1).}{fRL}
We see that $\partial^{0}f = f$  and we complement the definition by $\partial^{1}f:=f'.$

By $\da f $ we denote the fractional Caputo derivative $\da f(t)= \ra [f(\cdot)- f(0)](t)$, where $\alpha \in [0, 1]$. Then, for a nonnegative function $\mu: [0,1]\longrightarrow \mathbb{R}$ we define the distributed order Caputo derivative
\eqq{\dm f (t) = \int_{0}^{1}(\da f)(t) \mu(\al) d\al.}{disCap}

The problem with distributed order Caputo derivative was studied in \cite{Kochubei}. In that paper a fundamental solution to the Cauchy problem for the  equation $\dm u = \lap u $ is obtained. Then,  under the assumptions
\[
\mu \in C^{2}[0,1], \hd \mu(\al) = \al^{\nu}\mu_{1}(\al), \hd  \mu_{1}(\al)\geq \rho>0 \hd \m{ for } \hd  \al\in [0,1], \nu\geq 0 ,
\]
the estimates for the fundamental solution were proven and the formula for solution to initial value problem was obtained. Next, nonhomogeneous equation is considered under the assumption that a source term $f$ is continuous in $t$, bounded and locally H\"older continuous in $x$, uniformly with respect to $t$ (theorem~5.3 \cite{Kochubei}).

In the paper \cite{Luchko1} the equation
\[
\dm u = \divv(p(x)\nabla u)- q(x)u
\]
is studied, where $p\in C^{1}(\overline{\Omega})$, $q\in C(\overline{\Omega})$ and $p(x)>0$, $q(x) \geq 0$ for $x \in \overline{\Omega}$. Under the assumptions that $\mu \geq 0$, $\mu \not \equiv 0$ and $\mu$ is continuous the maximum principle and uniqueness of solution is proved. In \cite{cmf} the equation $\dm u = A u$ is considered, where $A$ is a generator of a bounded $C^{0}$-semigroup and $\mu $ belongs to $C^{3}[0,1]$, \hd $\mu(0)\not = 0$ and $\mu(1)\not = 0$ or $\mu(\alpha)= a\al^{\nu}$ as  $\al \rightarrow 0$, where $a,\nu>0$. In \cite{Li} the fractional diffusion equation with distributed Caputo derivative is analyzed, where the elliptic operator has time-independent coefficients. In that paper, under the assumptions that $\mu$ is a non-negative continuous function such that $\mu(0)\not = 0$ the asymptotic behavior of solution is described. In \cite{Rundell} a similar problem is studied, however it is assumed that $\mu $ is in $C^{1}[0,1]$,  $\mu \geq 0$ and $\mu(1) \not = 0$. We would like to emphasize that all above assumptions are  particular cases of our general assumption concerning the weight $\mu$. Furthermore, under the last condition (\ref{intmu}) holds. Finally, we recall the paper \cite{Zacher}, where the notion of $\mathscr{PC}$ pair was introduced (see definition 2.1 \cite{Zacher}). In theorem~\ref{fint} we show that for any non-negative nontrivial function $\mu$ from $L^{1}(0,1)$ and $k$ defined by (\ref{k}) there exists $g$ such that $(k,g)$ is $\mathscr{PC}$ pair. However we do not intend to apply the result from \cite{Zacher} and we present here an alternative argument, which in our opinion is more flexible and allows to deal with problems in non-cylindrical domains (see \cite{Stefan}).

This paper is devoted to the parabolic-type problem with the distributed order Caputo derivative with the density $\mu$. We only assume that
\eqq{\mu \in L^{1}(0,1), \hd \hd \mu \geq 0, \hd \hd \mu \not \equiv 0.}{mainmu}
\no Now we formulate the parabolic-type problem:  assume that $\Om \subseteq \mr^{N}$ is an open and bounded set with smooth boundary  and $N \geq 2.$ We will consider the following problem
\begin{equation}\label{uklad}
 \left\{ \begin{array}{ll}
\dm u = Lu+f & \textrm{ in } \Om \times (0,T)=: \Om^{T} \\
 u|_{\p \Om} = 0 & \textrm{ for  } t \in (0,T) \\
 u|_{t=0} =\uz & \textrm{ in } \Om, \\
\end{array} \right. \end{equation}
where
\[
Lu(x,t)\hspace{-0.1cm} =\hspace{-0.2cm} \sum_{i,j=1}^{N} D_{i}(a_{i,j}(x,t)D_{j}u(x,t)) + \sum_{j=1}^{N} b_{j}(x,t)D_{j}u(x,t) + c(x,t)u(x,t)
\]
and $b_{j} \in L^{2}(\Om^{T}),$ $c \in L^{2}(\Om^{T})$, $a_{i,j}$ are measurable and  $a_{i,j}= a_{j,i}.$ What is more, we assume that $L$ is uniformly elliptic, i.e. there exist positive constants $\lambda, \Lambda$ such, that
\eqq{
\lambda|\xi|^{2} \leq \sum_{i,j}^{N}a_{i,j}(x,t)\xi_{i}\xi_{j} \leq \Lambda|\xi|^{2} \hd \hd \m{ for a.a. } (x,t) \in \Om \times (0,T) \hd \forall \xi \in \mr^{N}.
}{elipt}
To formulate the main result we need the notion of weak solution to the problem (\ref{uklad}).
\begin{de}\label{weakdef}
We say that function $u \in L^{2}(0,T;H^{1}_{0}(\Om))$ such, that
\[
\int_{0}^{1}  I^{1-\al}[u-\uz] \mu(\al) d\al \in {}_{0}H^{1}(0,T;H^{-1}(\Om))
\]
 is a weak solution to the problem (\ref{uklad}), if for all $\vf \in H^{1}_{0}(\Om) $ and a.a. $t \in (0,T)$ $u$ fulfills the equality
\[
\poch \izj  \int_{\Om} I^{1-\al}[u(x,t)-\uz(x)]\vf(x)dx \ma d\al
\]
\[
 + \sum_{i,j=1}^{N} \int_{\Om}a_{i,j}(x,t)D_{j}u(x,t)D_{i}\vf(x)dx = \sum_{j=1}^{N} \int_{\Om} b_{j}(x,t)D_{j}u(x,t) \vf(x)dx
\]
\eqq{ + \int_{\Om} c(x,t)u(x,t)\vf(x)dx + \left\langle f(\cdot,t),\vf(\cdot)\right\rangle_{H^{-1} \times H_{0}^{1}(\Omega)}.}{slabadef}
\end{de}
\vspace{-0.2cm}
\no  We shall present our main result. Let us denote $b = (b_{1}, \dots, b_{N}).$
\begin{theorem}\label{tw1}
Suppose that $\mu $ satisfies (\ref{mainmu}),  $T>0$, $u_{0}\in L^{2}(\Omega)$ and $f\in L^{2}(0,T;H^{-1}(\Omega))$. Assume that  (\ref{elipt}) holds  and for some $p_{1},p_{2}\in [2,\frac{2N}{N-2} )$ we have $b \in L^{\infty}(0,T;L^{\frac{2p_{1}}{p_{1}-2}}(\Om))$, \hd $c\in  L^{\infty}(0,T;L^{\frac{p_{2}}{p_{2}-2}}(\Om))$.
Then there exists a unique  $u$  weak solution to (\ref{uklad}) in the sense of definition \ref{weakdef} and $u$ satisfies  the following  estimate
\[
\norm{ \int_{0}^{1}  I^{1-\al}[u-\uz] \mu(\al) d\al }_{H^{1}(0,T;H^{-1}(\Omega))} + \| u \|_{L^{2}(0,T;H^{1}_{0}(\Omega))}
\]
\eqq{ \leq  c_{0} \left(\| u_{0} \|_{L^{2}(\Omega)}+ \| f \|_{L^{2}(0,T;H^{-1}(\Omega))}\right), }{a12}
where $c_{0}$ depends only on, $\mu$, $\Omega$,  $\lambda$, $\Lambda$, $p_{1}$, $p_{2}$, $T,$ $\| b \|_{L^{\infty}(0,T;L^{\frac{2p_{1}}{p_{1}-2}}(\Omega) )}$, $\| c \|_{L^{\infty}(0,T;L^{\frac{p_{2}}{p_{2}-2}}(\Omega) )}$.
What is more, if
\eqq{
\int_{\frac{1}{2}}^{1}\ma d\al > 0,
}{intmu}
then $u \in C([0,T];H^{-1}(\Om))$ and $u|_{t=0} = \uz.$
\end{theorem}

\no If the  assumption (\ref{intmu}) of the theorem~\ref{tw1} is not satisfied, then we are also able to show continuity of $u$ in the case $L = \Delta$, but we have to assume more regularity about $f$. To be more precise, we introduce the notation
\eqq{\hkk= \left\{w \in H^{k}(\Omega): \hd \Delta^{l}w_{|\partial \Omega}=0  \m{ for } l=0,1,\dots, \left[ \frac{k-1}{2} \right]  \right\}}{e2}
\no and $\hkks$ denotes the dual space to $\hkk$. Suppose that (\ref{intmu}) does not hold. Then, by the assumption (\ref{mainmu}), there exists an uniquely determined number  $m  \in \mathbb{N} \setminus \{0\}$ such that
\eqq{
\int_{\frac{1}{2m}}^{1} \ma d\al = 0 \textrm{ \hd and \hd  } \int_{\frac{1}{2(m+1)}}^{\frac{1}{2m}} \ma d\al > 0.
}{intmum}

\begin{theorem}\label{tw2}
 Suppose that $\mu$ satisfies (\ref{mainmu}), $f \in L^{2}(0,T;H^{-1}(\Om)),$ $\uz \in L^{2}(\Om)$, and $u$ is a weak solution to (\ref{uklad}) given by theorem \ref{tw1} for $L=\lap$. Assume that (\ref{intmu}) does not hold and $m$ is given by (\ref{intmum}). If additionally for every $k=1,\cdots, m$
\[\int_{0}^{\frac{1}{2m}} \hspace{-0.2cm} \cdots \int_{0}^{\frac{1}{2m}} I^{1-(\al_{1}+\cdots+ \al_{k})}f(x,t) \prod_{i=1}^{k}\mu(\al_{i})d\al_{i} \in {}_{0}{W^{1,2}(0,T;(\bar{H}^{2k+1})^{\ast})},
\]
then $u \in C([0,T];(\bar{H}^{2m+1})^{\ast})$ and $u|_{t=0} = \uz$ in $(\bar{H}^{2m+1})^{\ast}.$
\end{theorem}

\no Now we formulate the result concerning more regular solutions.

\begin{theorem}
Suppose that $\mu$ satisfies (\ref{mainmu}), $f\in L^{2}(0,T;\ld )$, $u_{0}\in \hjzo, $  (\ref{elipt}) holds,  $\max_{i,j}\| \nabla \aij \|_{L^{\infty}(\OT)}<\infty$ and
for some $p_{1}\in  [2,\frac{2N}{N-2} )$, $p_{2}\in [2,\frac{2N}{N-2} )\cap [2,4]$ we have $b \in L^{\infty}
(0,T;L^{\frac{2p_{1}}{p_{1}-2}}(\Omega))$,
\hd $c\in L^{\infty}(0,T;L^{\frac{p_{2}}{p_{2}-2}}(\Omega))$.
Then the problem (\ref{uklad}) has exactly one  solution
$u $ in $ L^{2}(0,T;H^{2}(\Omega))$
such that $\izj \ija [u-u_{0}] \ma d\al \in {}_{0}H^{1}(0,T;\ld)$ and  the following estimate
\[
\| u \|_{L^{2}(0,T;H^{2}(\Omega))} +
\left\| \izj  \ija [u-u_{0}] \hj \ma d\al \right\|_{H^{1}(0,T;\ld)}
\]
\eqq{
 \leq  C_{0}(\| u_{0}\|_{\hjzo }
+ \| f \|_{L^{2}(0,T;\ld)} )
}{b5}
holds, where $C_{0}$ depends only on $\mu$, $\lambda$, $\mu$, $p_{1}$,
$p_{2}$,  $T$,  $\| \nabla \aij \|_{L^{\infty}(\OT)}$, the Poincar\'e constant
and the $C^{2}$-regularity of $\partial \Omega$ and the norms
$ \| b \|_{L^{\infty}(0,T; L^{\frac{2p_{1}}{p_{1}-2}}(\Omega))}$,
$ \| c \|_{L^{\infty}(0,T; L^{\frac{p_{2}}{p_{2}-2}}(\Omega))}$.

\no Furthermore, if (\ref{intmu}) holds, then $u\in C([0,T];\ld)$ and $u_{|t=0}= u_{0}$.
\label{regularne}
\end{theorem}

\no The last theorem is devoted to examination of the right-inverse operator to $\dm$.

\begin{theorem}
If $\mu$ satisfies (\ref{mainmu}), then there exists nonnegative $g\in L^{1}_{loc}[0,\infty)$ such that the operator of fractional integration $\im$, defined by the formula $ \im u =g*u$ satisfies
\eqq{(\dm \im u)(t) = u(t) \hd \m{ for } \hd  u \in L^{\infty}(0,T),}{fi1}
\eqq{(\im\dm u)(t) = u(t) - u(0) \hd \m{ for } \hd  u \in AC[0,T].}{fi2}
Furthermore, $g$ satisfies the estimate
\eqq{g(t)\leq c\max\{t^{\gamma-1},t^{-\gamma} \}}{estig}
for some positive $c$ and $\gamma\in (0,\frac{1}{2})$.
\label{fint}
\end{theorem}
As we will see  later, the constant $\gamma$ is taken from (\ref{defg}). The  properties of the fractional integration operator $\im$ related to  the distributed order Caputo derivative will be presented in section~\ref{fintsection}.

The paper is organized as follows. In second section we prove theorem~\ref{fint}. In next three sections we prove theorems~\ref{tw1},  \ref{tw2} and \ref{regularne}, respectively.

\section{The fractional integration operator $\im$}

\label{fintsection}

In this section we will define the fractional integration operator $\im$ and prove the useful properties of its kernel $g$. First we note that the distributed order Caputo derivative is well defined for absolutely continuous functions and we have
\eqq{\m{ if } \hd  f\in AC[0,T], \m{ \hd then \hd } \dm f \in L^{1}(0,T). }{Cac}
Indeed, for $f \in AC[0,T]$  we write
\[
(\dm f)(t) = (k*f')(t) = \poch (k*f)(t) - k(t)f(0),
\]
where
\eqq{k(t) = \izj  t^{-\al} \mg d \al. }{k}
and $*$ here and in the whole paper denotes the convolution on $(0,\infty)$, i.e.  $(k*f)(t)=\izt k(t-\tau)f(\tau)d\tau.$ Then  $k \in L^{1}(0,T)$, because using the fact that  $\frac{1}{2} \leq \Gamma(x) $ on $[1,2]$ we have
\[
\int_{0}^{T}\abs{\izj  t^{-\al}\mg d\al}dt = \int_{0}^{T}\izj t^{-\al} \mg d\al dt
\]
\eqq{
=  \izj T^{1-\al} \frac{\ma}{\Gamma(2-\al)}d\al \leq 2 \max\{1,T\} \izj \mu(\alpha) \ald <\infty.
}{xc}
By Young's theorem we obtain that $\dm f\in L^{1}(0,T).$

The main assumption (\ref{mainmu}) has the~following consequence: if we denote
\eqq{
c_{\mu} =\izj \ma d\al  > 0,
}{aab}
then
\eqq{
\exists  \gamma \in (0, \frac{1}{2})  \hd \  \int_{\gamma}^{1-\gamma} \ma d\al = \frac{1-\gamma}{2} c_{\mu} > 0.
}{defg}
This statement easily follows from Darboux theorem applicated to function $h:[0,\frac{1}{2}]\rightarrow \mr$ defined by
\[
h(x) = \int_{x}^{1-x}\ma d\al - \frac{1-x}{2}\izj \ma d\al.
\]
The constant $\gamma$ related to $\mu$ will play the crucial role in our considerations.

It is already known, see for example \cite{Kochubei},\cite{Kochubei2},\cite{Kochubei3}, that under some assumptions concerning $\mu$, there exists the right  inverse  operator to the distributed order  Caputo derivative. It means that there exists the kernel $g$ such that if we denote by $\im u = g*u,$ then for $u$ regular enough we have $(\dm\im u)(t) = u(t)$.
The properties of the kernel $g$ has been already studied in above mentioned papers. For example, in \cite{Kochubei}, it has been shown that if
\[
\mu \in C^{3}[0,1], \ \  \mu(1)\neq 0 \textrm{ \hd and either \hd  }\mu(0) \neq 0\textrm{ \hd or \hd }\ma \approx a \al^{\nu}
\]
for some $a > 0, \nu > 0$ then
\eqq{g(t)\hspace{-0.1cm} = \hspace{-0.1cm}\frac{1}{\pi}\izi e^{-rt}\frac{\izj \sin (\pi \al)r^{\al}\mu(\al)d\al}{\left(\izj \cos (\pi \al)r^{\al}\mu(\al)d\al\right)^{2}+\left(\izj \sin (\pi \al)r^{\al}\mu(\al)d\al\right)^{2}}dr}{calka}
and for small values of $t$
\[
g(t) \leq c \ln \frac{1}{t}.
\]
The last estimate  implies that for all $p \in (1,\infty)$ we have $g(t) \in L^{p}_{loc}[0,\infty).$
We will obtain similar result to the above mentioned one, but for $\mu$ with much lower regularity.

We will divide the proof of theorem~\ref{fint} into a few steps.
Firstly we will proceed as in \cite{Kochubei}, i.e. we investigate the Laplace transform of function $k$, given by (\ref{k}),
\[
\tl{k}(p) = \izj p^{\al -1 }\mu(\al)d\al.
\]
The first step is the following proposition.
\begin{prop}
For the Laplace transform of $g$ given by (\ref{calka}) we have
\eqq{
\tl{g}(p) =\frac{1}{p\tl{k}(p)},
}{lapg}
and so $k*g = 1$.
\label{propone}
\end{prop}
The main difficulty is to obtain the equality (\ref{lapg}) only under the assumption (\ref{mainmu}). The proof of proposition~\ref{propone} will base on  Lemma \ref{odwrotna} from the appendix.
\begin{proof}[Proof of proposition~\ref{propone}]
We shall show that (\ref{lapg}) follows from lemma~\ref{odwrotna}. Indeed, we denote  $F(p):=\frac{1}{p\tl{k}(p)}$ and we prove  that $F(p)$ satisfies assumptions of Lemma \ref{odwrotna}.

\no \textit{Assumption 1).}
Let us denote by $r = |p|,$ $\vf=\arg p$. Choosing the main branch of logarithm we see that
\[
p \tl{k}(p)= \izj p^{\al}\mu(\al)d\al = \izj r^{\al}\cos(\al\vf)\ma d\al + i\izj r^{\al}\sin(\al\vf)\ma d\al
\]
is analytic for $p \in \mathbb{C} \setminus (-\infty,0].$ The imaginary part is nonzero for  $|\vf|\in (0,\pi) $ and for $\vf =0$ the real part is positive. Thus we have
\[
p \tl{k}(p)=\izj p^{\al}\mu(\al)d\al \not = 0 \m{ \hd for \hd } p \in \mathbb{C} \setminus (-\infty,0],
\]
and $F(p)=\frac{1}{p\tl{k}(p)}$ is analytic for  $p \in \mathbb{C} \setminus (-\infty,0].$

\no \textit{Assumtion 2).} For $t>0$ the quantity $\izj t^{\al}\sin(\al\pi)\ma d\al$ is positive, hence $F^{\pm}(t)= \lim\limits_{\vf\rightarrow \pi^{-}}F(te^{\pm i \vf})$ exists. By equality
\[
\lim_{\vf\rightarrow \pi^{-}} r^{\al}e^{i \vf \al}= r^{\al} e^{i\al\pi} = \overline{r^{\al} e^{-i\al\pi}} = \overline{ \lim_{\vf\rightarrow \pi^{-}}r^{\al}e^{-i \vf \al}}.
\]
and Lebesgue dominated convergence theorem we  have  $F^{+} = \overline{F^{-}}$.

\no \textit{Assumtion 3).} Firstly we have to estimate $\abs{\izj p^{\al}\mu (\al)d\al}$ from below. We may assume that $\eta \in (0,\frac{\pi}{2})$.   If $|\arg(p)| <\pi - \eta$, then we write
$p = r[\cos(\pm\vf) + i \sin (\pm\vf)]$ for $\vf \in [0,\pi - \eta ) $. We have to consider two cases. If $\vf \in [0,\frac{\pi}{2}]$, then
\[
\abs{\izj p^{\al}\mu(\al)d\al}
\geq \abs{\izj r^{\al}\cos(\pm\vf\al)\mu(\al)d\al}  = \izj r^{\al}\cos(\vf\al)\mu(\al)d\al
\]
\[
\geq  \int_{\gamma}^{1-\gamma} r^{\al}\cos(\vf\al)\mu(\al)d\al \geq \cos{\frac{\pi}{2}(1-\gamma)} \int_{\gamma}^{1-\gamma} r^{\al}\mu(\al)d\al,
\]
where  $\gamma$ comes from (\ref{defg}). If  $\vf \in (\frac{\pi}{2},\pi- \eta)$ then  we have
\[
\abs{\izj p^{\al}\mu(\al)d\al}
\geq \abs{\izj r^{\al}\sin(\pm\vf\al)\mu(\al)d\al}
\]
\[
 = \izj r^{\al}\sin(\vf\al)\mu(\al)d\al
\geq \int_{\gamma}^{1-\gamma}r^{\al}\sin(\vf\al)\ma d\al
\]
Having in mind that $\gamma \in (0,\frac{1}{2}),$ we obtain the estimate
\[
\sin \gamma\vf \geq \sin \gamma\frac{\pi}{2}, \hd  \hd
\sin(1-\gamma)\vf \geq \min\{\sin{\gamma\pi},\frac{\sqrt{2}}{2}\}.
\]
Hence, if $\hd  | \arg(p)|<\pi - \eta$, we have
\[
\abs{\izj p^{\al}\mu(\al)d\al} \geq \min\left\{ \sin{\gamma \frac{\pi}{2}},\sin{\gamma\pi},\frac{\sqrt{2}}{2} \right\} \int_{\gamma}^{1-\gamma}r^{\al}\ma d\al.
\]
Thus after applying (\ref{defg}) and  under the condition $|\arg{p}|\leq \pi - \eta $   we get   the following estimates
\eqq{\abs{\izj p^{\al}\mu(\al)d\al} \geq \min\left\{ \sin{\gamma \frac{\pi}{2}},\sin{\gamma\pi},\frac{\sqrt{2}}{2} \right\} \frac{1-\gamma}{2}c_{\mu}  r^{1-\gamma} \equiv \tilde{c}r^{1-\gamma}}{kpm}
for $r\leq 1$ and
\eqq{
\abs{\izj p^{\al}\mu(\al)d\al} \geq \tilde{c}  r^{\gamma}\textrm{ \hd  for } \hd r > 1,}{kpd}
where $r=|p|$ and  constant $\tilde{c}$ depends only on $\mu$.

\no  Using (\ref{kpd}) we obtain that for $p = r e^{i\pm\vf}$ and $\phi \in [0,\pi - \eta)$
\[
\abs{F(p)}=\frac{1}{\abs{p\tl{k}(p)}}=\frac{1}{\abs{\izj p^{\al}\mu(\al)d\al}} \leq \frac{1}{\tilde{c}} r^{-\gamma} \longrightarrow 0 \textrm{ \hd as \hd  } r\rightarrow \infty.
\]
Similarly, using (\ref{kpm})  we get
\[
\abs{p}\abs{F(p)}=\frac{\abs{p}}{\abs{p\tl{k}(p)}}=\frac{r}{\abs{\izj p^{\al}\mu(\al)d\al}} \leq \frac{1}{\tilde{c}} r^{\gamma} \longrightarrow 0  \textrm{ \hd as \hd   } r\rightarrow 0.
\]
Thus the third assumption is satisfied with any $\eta \in (0, \frac{\pi}{2})$.

\no \textit{Assumtion 4). } We denote
\[
a(r) = \left\{ \begin{array}{ll}
 \frac{1}{\tilde{c}} r^{\gamma-1}& \textrm{ for }  \hd r \leq 1 \\
 \frac{1}{\tilde{c}}	r^{-\gamma} &\textrm{ for } \hd r > 1 \\
\end{array} \right.,
\]
where $\tilde{c}=\tilde{c}(\mu)$ is as above. Then from (\ref{kpm}) and (\ref{kpd}) we have  $|F(p)| \leq a(|p|)$, where  $p = |p|e^{i\phi}$ and  $\phi \in (\frac{\pi}{2}, \pi)$. Using the fact that  $\gamma $ belongs to  $(0,1)$ we deduce that  function $a(r)$ satisfies the desired condition.

We have just shown that $F(p)=\frac{1}{p\tl{k}(p)}$ satisfies assumptions of lemma \ref{odwrotna}, thus we can write that
\[
\frac{1}{p\tl{k}(p)} = \izi e^{-pt} g(t)dt,
\]
where
\[
g(t) = \frac{1}{\pi}\izi e^{-rt} \Im\left[\izj r^{\al}e^{-i\pi\al}\ma d\al \right]^{-1}dr
\]
\[
=\frac{1}{\pi}\izi e^{-rt}\frac{\izj \sin (\pi \al)r^{\al}\mu(\al)d\al}{\left(\izj \cos (\pi \al)r^{\al}\mu(\al)d\al\right)^{2}+\left(\izj \sin (\pi \al)r^{\al}\mu(\al)d\al\right)^{2}}dr.
\]
Thus we obtained that $\frac{1}{p\tl{k}(p)} = \tl{g}(p),$ which implies that $g*k = 1$ and the proof of proposition is completed.
\end{proof}

\begin{prop}\label{oszacg1}
The function $g$ given by (\ref{calka}) satisfies (\ref{estig}) with $\gamma\in (0, \frac{1}{2})$ given by (\ref{defg}) and $c$ depends only on $\mu$.
\end{prop}

\begin{proof}
Ignoring expression with cosine in denominator we get that
\[
g(t) \leq  \frac{1}{\pi}\izi e^{-rt}\frac{1}{\izj \sin (\pi \al)r^{\al}\mu(\al)d\al}dr.
\]
The expression under integral is non negative, thus if we take $\gamma$ from  (\ref{defg}), then  for $r\leq 1$ we can write
\[
\izj \sin (\pi \al)r^{\al}\mu(\al)d\al \hspace{-0.1cm} \geq \hspace{-0.1cm} \int_{\gamma}^{1-\gamma} \sin (\pi \al)r^{\al}\mu(\al)d\al \geq \frac{1-\gamma}{2} c_{\mu} \sin(\pi\gamma) r^{1-\gamma}\hspace{-0.3cm}.
\]
Using the identity
\eqq{\frac{\pi}{\sin\pi\gamma} = \ggg\gjg }{gsin}
we obtain that for $r\leq 1$
\[
\frac{1}{\izj \sin (\pi \al)r^{\al}\mu(\al)d\al} \leq \frac{2\ggg\gjg}{\pi (1-\gamma)c_{\mu}} r^{\gamma-1}.
\]
Similarly for $r>1$
\[
\izj \sin (\pi \al)r^{\al}\mu(\al)d\al \geq \int_{\gamma}^{1-\gamma} \sin (\pi \al)r^{\al}\mu(\al)d\al \geq \frac{1-\gamma}{2} c_{\mu}\sin(\pi\gamma) r^{\gamma} ,
\]
thus from (\ref{gsin}) for $r>1$
\[
\frac{1}{\izj \sin (\pi \al)r^{\al}\mu(\al)d\al} \leq \frac{2\ggg\gjg}{\pi (1-\gamma)c_{\mu}} r^{-\gamma}
\]
and we arrive at
\[
g(t) \leq \frac{2\ggg\gjg}{\pi^2 (1-\gamma)c_{\mu}}\left(\int_{1}^{\infty} e^{-rt} r^{-\gamma} dr + \izj e^{-rt} r^{\gamma-1}dr\right).
\]
With the first integral  we deal as follows
\[
\int_{1}^{\infty} e^{-rt} r^{-\gamma} dr \leq \izi e^{-rt} r^{-\gamma} dr = \gjg t^{\gamma-1}.
\]
Similarly we estimate  the second integral
\[
\izj e^{-rt} r^{\gamma-1}dr \leq \izi e^{-rt} r^{\gamma-1}dr = \Gamma(\gamma) t^{-\gamma}.
\]
Thus we have
\[
g(t) \leq \frac{2\ggg\gjg}{\pi^2 (1-\gamma)c_{\mu}}\left[\gjg t^{\gamma-1} + \Gamma(\gamma) t^{-\gamma} \right].
\]
Having in mind that $1-\gamma > \frac{1}{2}$ we get $\gjg \leq \Gamma(\frac{1}{2})=\sqrt{\pi}$ we obtain
\[
g(t) \leq \frac{4\ggg\sqrt{\pi}}{\pi^2 c_{\mu}}\left[\sqrt{\pi} t^{\gamma-1} + \Gamma(\gamma) t^{-\gamma} \right],
\]
and the  proof is finished.
\end{proof}

\begin{prop}
Function $g$ defined by (\ref{calka}) is continuous on $(0,\infty)$ and for $\gamma$ given by (\ref{defg}) we have
\eqq{\lim_{t\rightarrow 0^{+} }t^{1-\gamma } g(t)=0.  }{limgz}
In particular, function $t^{1-\gamma } g(t)$ is continuous on $[0, \infty)$.
\label{contg}
\end{prop}

\begin{proof}
If $t_{0}\in (0,\infty)$, then  taking advantage of the  monotonicity of $e^{-rt_{0}}-e^{-rt}=e^{-rt_{0}} [1-e^{-r(t-t_{0})}]$ with respect to $t$ we may apply Lebesgue monotone convergence theorem and we have $\lim_{t\rightarrow t_{0}}g(t)=g(t_{0})$.

\no Now we show (\ref{limgz}). Arguing similarly as in derivation of $\gamma$ we see that there exists  $\gs\in (\gamma, \frac{1}{2})$ such that
\[
  \int_{\gs}^{1-\gs} \ma d\al = \frac{1-\gamma}{4} c_{\mu} > 0.
\]
Repeating the proof of proposition~\ref{oszacg1} for $\gs$ we get
\[
g(t) \leq \frac{8\Gamma(\gs)\sqrt{\pi}}{\pi^2 c_{\mu}}\left[\sqrt{\pi} t^{\gs-1} + \Gamma(\gs) t^{-\gs} \right],
\]
which gives (\ref{limgz}).

\end{proof}

\begin{proof}[Proof of theorem~\ref{fint}]
By proposition~\ref{oszacg1} function $g$ defined by (\ref{calka}) belongs to $L^{1}_{loc}[0,\infty)$ and satisfies (\ref{estig}). It remains to show (\ref{fi1}) and (\ref{fi2}) for $\im $ defined as a convolution with $g$. Assume that $f \in L^{\infty}(0,T).$ If we recall the definition of kernel $k$ (\ref{k}) we see that
\[
(\dm \im f)(t) = (\dm g*f)(t) = \poch k*g*f - k(t)(g*f)(0).
\]
The last therm is equal to zero, because $g \in L^{1}(0,T)$ and we can estimate as follows
\[
\abs{(g*f)(t)} \leq \izt \abs{g(\tau)f(t-\tau)}d\tau \leq \norm{f}_{L^{\infty}(0,t)}\norm{g}_{L^{1}(0,t)} \rightarrow 0 \textrm{ as } t\rightarrow 0.
\]
Thus using proposition~\ref{propone} we get $(\dm \im f)(t) = \poch \izt f(\tau)d\tau = f(t)$ a.e on $(0,T).$\\
Assume now that $f \in AC[0,T].$ Then using again the equality $g*k=1$ we have
\[
(\im \dm f)(t) = \left(g*\left[\frac{d}{d\tau} k * f - k(\tau)f(0)\right]\right)(t)
\]
\[
= \left(g*\left[\frac{d}{d\tau} k * f\right]\right)(t) - f(0)(g*k)(t)
\]
\[
= \izt g(\tau) \frac{d}{dt}\int_{0}^{t-\tau}k(t-\tau-s)f(s)dsd\tau - f(0)
\]
\[
= \poch\izt  g(\tau)\int_{0}^{t-\tau} k(t-\tau-s)f(s)dsd\tau  - f(0)
\]
\[
 = \poch(g*k*f)(t)-f(0) = f(t) - f(0).
\]
\end{proof}

\begin{coro}\label{oszacg12}
Since $g$ is non increasing, from estimate (\ref{estig}) immediately follows that for fixed $T>0$
\[
g(t) \leq \cmt t^{\gamma-1} \textrm{ \hd  for \hd  } 0\leq t \leq T,
\]
where constant $\cmt $ depends only on $\mu$ and $T$.
\end{coro}

Using estimates for the kernel $g$ we are able to give a simple proof of abstract Gronwall lemma in the case of the distributed order  Caputo derivative.

\begin{lem}\label{gronwall}
Assume that $T \in (0,\infty)$ and $a,w$ are non negative functions, integrable on $(0,T).$ Let $f$ be non decreasing and bounded on $(0,T)$. Then, if  $w$ satisfies inequality
\eqq{
w(t) \leq a(t) + f(t)(g*w)(t) \textrm{ \hd for \hd  } t \in (0,T],}{gronz}
then
\eqq{
w(t) \leq \sum_{k=0}^{\infty}f^{k}(t) (g^{[k]}*a)(t) \textrm{ \hd  for \hd  } t \in (0,T],}{gront}
where  we denote by $(g^{[k]}*a)(t):=(g*\cdots*g*a)(t),$ where $g$ is taken k times. What is more, the series above is convergent for all  $t \in (0,T].$\\
\end{lem}

\begin{proof}
We convolve (\ref{gronz}) with $g$ and multiply by $f$.
\[
f(t)(g*w)(t) \leq f(t)(g*a)(t) + f(t)g*\left(f(\cdot)(g*w)(\cdot)\right)(t)
\]
\[
 \leq f(t)(g*a)(t) + f^{2}(t)(g^{[2]}*w)(t).
\]
Using this inequality in (\ref{gronz}) we obtain that
\[
w(t) \leq a(t) + f(t)(g*a)(t) + f^{2}(t)(g^{[2]}*w)(t).
\]
If we continue this procedure, we arrive at
\eqq{
w(t) \leq \sum_{k=0}^{n} f^{k}(t)(g^{[k]}*a)(t) + f^{n}(t)(g^{[n]}*w)(t)}{gronn}
We will show that for all $t\in (0,T]$ $(g^{[n]}*w)(t) \rightarrow 0$ as $n\rightarrow \infty.$
Recall, that from corollary~\ref{oszacg12} for all $t\in (0,T]$
\[
g(t) \leq \cmt t^{\gamma-1}.
\]
From definition (\ref{fI}) we get
\[
(g*w)(t) \leq \cmt \izt (t-\tau)^{\gamma-1}w(\tau)d\tau = \cmt\ggg (I^{\gamma}w)(t).
\]
Applying  the convolution with $g$ to both sides of this inequality  and using identity $I^{\al}I^{\beta}u = I^{\al+\beta}u$ for $\al,\beta > 0$ (see (2.21) in \cite{Samko}) we get that
\[
(g*g*w)(t) \leq \cmt^{2}\ggg \izt (t-\tau)^{\gamma-1}(I^{\gamma}w)(\tau)d\tau
= \cmt^{2} (\ggg)^{2} (I^{2\gamma}w)(t).
\]
Proceeding this way we obtain the estimate
\eqq{
(g^{[n]}*w)(t) \leq \cmt^{n}(\ggg)^{n}(I^{n\gamma}w)(t).
}{groncn}
Thus for $n$ such that $n \gamma >1$ we have
\[
(g^{[n]}*w)(t) \leq \cmt^{n}\frac{(\ggg)^{n}}{\Gamma(n\gamma)}\izt (t-\tau)^{n\gamma-1}w(\tau)d\tau
\]
\eqq{
\leq \cmt^{n}\frac{(\ggg)^{n}}{\Gamma(n\gamma)} t^{n\gamma-1}\norm{w}_{L^{1}(0,t)}.
}{gronc}
Due to a presence of expression $\Gamma(n\gamma)$ in the denominator, the last phrase tends to zero as $n\rightarrow \infty$ uniformly for all $t \in (0,T] $.

We will show that for $t \in (0,T]$ the  series in (\ref{gront}) is convergent. Applying (\ref{gronc}) we have that
\[
\sum_{k=0}^{\infty}f^{k}(t)(g^{[k]}*a)(t)
\]
\eqq{
 \leq \hspace{-0.2cm} \sum_{ \{ k:  \/ \gamma k\leq 1 \} }f^{k}(t)(g^{[k]}*a)(t)+  \sum_{ \{ k: \/ \gamma k >1 \} }f^{k}(t) \cmt^{k}\frac{(\ggg)^{k}}{\Gamma(k\gamma)} t^{k\gamma-1}\norm{a}_{L^{1}(0,t)}.
}{grone}
Denoting the expression under the sum by $b_{k}$ we see that
\[
\frac{b_{k+1}}{b_{k}} = \frac{f(t)\cmt \ggg t^{\gamma}\Gamma(k\gamma)}{\Gamma(k\gamma+\gamma)},
\]
thus
\[
\frac{b_{k+1}}{b_{k}} = f(t)\cmt \ggg t^{\gamma} B(\gamma, k \gamma) \rightarrow 0 \textrm{ as } k\rightarrow \infty,
\]
so from d'Alembert criterion for $ t \in (0,T]$ series from estimate (\ref{gront}) is convergent.
\end{proof}

\section{Weak solutions}

In this section we prove theorem~\ref{tw1}. The proof will be divided onto a few steps.  First we obtain approximate solutions and  energy estimates. Further, by weak compactness argument  we get a weak  solution  and after that we show its uniqueness. Finally, we derive the continuity  of solution.

\subsection{Approximate solution}
We denote by $\eta_{\ve}= \eta_{\ve}(t)$ the standard smoothing kernel with the support in $[-\ve,\ve]$ and we assume that $\eta_{\ve}(t)=\eta_{\ve}(-t)$.  By $\overline{\ast}$ we denote the convolution on real line. Then we set
\[
a_{i,j}^{n}(x,t) = \eta_{\frac{1}{n}}(\cdot)\overline{\ast}a_{i,j}(x,\cdot)(t),
\]
where we extend $a_{i,j}(t)$ by even reflection for $t \notin (0,T).$ Then $a_{i,j}^{n}\rightarrow a_{i,j}$ in $L^{2}(\Om^{T})$ and
\[
\lambda|\xi|^{2} \leq \sum_{i,j}^{N}a_{i,j}^{n}(x,t)\xi_{i}\xi_{j} \leq \Lambda|\xi|^{2} \textrm{  } \hd \forall t \in [0,T], \hd \forall \xi \in \mr^{N}, \m{ a.a. } x\in \Omega.
\]
Further, we extend  $b_{j}, c$ by zero beyond the interval $[0,T]$ and we extend $f$ by odd reflection to the interval $(-T,T)$ and we set zero elsewhere. We denote
\[
b^{n}_{j}(x,t) =\eta_{\frac{1}{n}}(\cdot)\overline{\ast}b_{j}(x,\cdot)(t), \ \ \ \ c^{n}(x,t)  =  \eta_{\frac{1}{n}}(\cdot)\overline{\ast}c(x,\cdot)(t),
\]
\[
 f_{\frac{1}{n}} = \eta_{\frac{1}{n}}(\cdot)\overline{\ast} f(x,\cdot)(t).
\]
We look for approximate solution to system (\ref{uklad}) in the form
\eqq{
u^{n}(x,t) = \sum_{k=1}^{n}c_{n,k}(t)\vf_{k}(x) \textrm{ for } n \in \mathbb{N}.}{un}
where $\{\vf_{n}\}_{n=1}^{\infty}$ forms orthonormal  basis of $L^{2}(\Om).$ Precisely $\vf_{n}$ fulfill
\begin{equation}\label{vf}
 \left\{ \begin{array}{rlll}
-\Delta \vf_{n} &=& \lambda_{n}\vf_{n} & \textrm{ in } \hd  \Om \\
 {\vf_{n}}_{| \p \Om } &=& 0. &  \\
\end{array} \right. \end{equation}
In order to find $c_{n,k},$ we shall consider the approximate problem of the form
\begin{equation}\label{nn}
 \left\{ \begin{array}{ll}
\dm u^{n}(x,t) = L^{n}u^{n}(x,t) +f^{n}(x,t)& \textrm{ in } \Om^{T} \\
 u^{n}(x,0) = \uz^{n}, & \\
\end{array} \right. \end{equation}
where
\[
f^{n}(x,t) = \sum_{k=1}^{n} \left\langle f_{\frac{1}{n}(y,t)},\vf_{k}(y)\right\rangle_{H^{-1} \times H_{0}^{1}}\vf_{k}(x),
\]
\[
\uz^{n} = \sum_{k=1}^{n}\int_{\Om}\uz(y)\vf_{k}(y)dy\vf_{k}(x),
\]
\[
 L^{n}u^{n}(x,t) = \sum_{i,j=1}^{N}D_{i}(a_{i,j}^{n}(x,t)D_{j}u^{n}(x,t))
\]
\[
 + \sum_{j=1}^{N} b^{n}_{j}(x,t)D_{j}u^{n}(x,t) + c^{n}(x,t)u^{n}(x,t).
\]
Making use of (\ref{un}) in (\ref{nn}), multiplying by $\vf_{m}$ for $m=1, \dots, n$ and integrating over $\Om,$ we get
\[
\dm \int_{\Om} \sum_{k=1}^{n}c_{n,k}(t)\vf_{k}(x)\vf_{m}(x)dx
\]
\[
 = -\sum_{i,j=1}^{N}\int_{\Om} a^{n}_{i,j}(x,t)\sum_{k=1}^{n}c_{n,k}(t)D_{j}\vf_{k}(x)D_{i}\vf_{m}(x)dx
\]
\[
+\sum_{j=1}^{N} \iO b^{n}_{j}(x,t)\sum_{k=1}^{n}c_{n,k}(t)D_{j}\vf_{k}(x)\vf_{m}(x)dx
\]
\eqq{
+ \iO c^{n}(x,t)\sum_{k=1}^{n}c_{n,k}(t)\vf_{k}(x)\vf_{m}(x)dx + \left\langle f_{\frac{1}{n}}(\cdot,t),\vf_{m}\right\rangle_{H^{-1} \times H_{0}^{1}}.
}{przyb}
Using the orthogonality of  $\{\vf_{n}\}_{n=1}^{\infty}$ we obtain
\begin{equation}
\left\{ \begin{array}{l} \label{cn}
\dm c_{n}(t) = -A^{n}(t)c_{n}(t) + F^{n}(t) \\
 c_{n}(0) = c_{n,0},\\
\end{array} \right. \end{equation}
where
\[
c_{n}(t) = \left(c_{n,1}(t),\cdots, c_{n,n}(t)\right),
\]
\[
A^{n}(t)=\Bigg(\iO\sum_{i,j=1}^{N}a_{i,j}^{n}(x,t)D_{j}\vf_{k}(x)D_{i}\vf_{m}(x)
\]
\[
   - \left[\sum_{j=1}^{N} b_{j}^{n}(x,t)D_{j}\vf_{k}(x) - c^{n}(x,t)\vf_{k}(x) \right]\vf_{m}(x)dx\Bigg)_{k,m=1}^{n}
\]
\[
F^{n}(t) = \left(\int_{\Om} f_{\frac{1}{n}}(y,t)\vf_{1}(y)dy, \cdots, f_{\frac{1}{n}}(y,t)\vf_{n}(y)dy\right),
\]
\[
c_{n,0} = \left(\int_{\Om} \uz(y)\vf_{1}(y)dy, \cdots, \int_{\Om}\uz(y)\vf_{n}(y)dy\right).
\]
From theorem~\ref{fint} we deduce, that for $c_{n}\in AC[0,T]$ the system (\ref{cn}) is equivalent to the following problem
\begin{equation}
 \left\{ \begin{array}{l} \label{cnc}
c_{n}(t) = c_{n}(0) - (g*A^{n}c_{n})(t) + (g*F^{n})(t) \\
 c_{n}(0) = c_{n,0}, \\
\end{array} \right. \end{equation}
which we obtain from (\ref{cn}) by applying $\im$ to both sides. Our aim is to show that the system (\ref{cnc}) has an absolutely continuous solution. For this purpose we use Banach fixed point theorem. We shall consider the following class of functions
\eqq{X(T) = \{ c \in C([0,T];\mr^{n}): c(0) = c_{n,0}, \ \  t^{1-\gamma} c'(t) \in C([0,T];\mr^{n})\},}{X}
where $\gamma$ is taken from (\ref{defg}). We define metrics on  $X(T)$ as
 \[
\norm{c}_{X(T)} = \norm{c}_{C[0,T]} + \norm{t^{1-\gamma} c'}_{C[0,T]}.
\]
Then, using standard argument we can show that  $X(T)$ is complete metric space.
\begin{lem}
For every $n \in \mathbb{N}$ and $T \in (0,\infty)$ there exists a unique solution to (\ref{cnc}) in $X(T)$.
\label{lemthree}
\end{lem}
\begin{proof}
It is enough to show that operator
$(Pc)(t) = c_{n,0} - (g*A^{n}c)(t) + (g*F^{n})(t)$ conducts elements from $X(T_{1})$ into $X(T_{1})$ for every $T_{1} \in (0,T]$ and is a contraction on $X(T_{1})$ for some $T_{1}$ small enough. Later we extend the solution on the whole interval $[0,T].$
Let $c \in X(T_{1})$ for some $T_{1}\in (0,T].$ Then $(Pc)(0) = c_{n,0}.$ To get $Pc \in X(T_{1})$ it is sufficient to show that  $t^{1-\gamma} (Pc)'$ is continuous on $[0,T_{1}]$. For this purpose we write
\[
t^{1-\gamma} (Pc)'(t)= \tjg g(t) (- A^{n}(0)c(0)+F^{n}(0) ) + \tjg g*h'(t),
\]
where $h(t)= -A^{n}(t)c(t)+F^{n}(t)$. By proposition~\ref{contg} we get the continuity of the first component. To deal with the second one we write
\[
\left|  t^{1- \gamma}_{2} \int_{0}^{t_{2}} g(\tau) h'(t_{2}- \tau) \dt - t^{1- \gamma}_{1} \int_{0}^{t_{1}} g(\tau) h'(t_{1}- \tau) \dt    \right|
\]
\[
\leq |t_{2}^{1-\gamma}- t_{1}^{1-\gamma}| \int_{0}^{t_{2}} g(\tau) |h'(t_{2}- \tau)| \dt + t_{1}^{1-\gamma}  \int_{t_{1}}^{t_{2}} g(\tau) |h'(t_{2}- \tau)| \dt
\]
\[
+ t_{1}^{1-\gamma}\hspace{-0.2cm}  \int_{0}^{t_{1}}\hspace{-0.2cm} g(\tau) |h'(t_{2}- \tau)-h'(t_{1}- \tau)| \dt \hspace{-0.1cm} \equiv \hspace{-0.1cm} M_{1}(t_{1}, t_{2})+M_{2}(t_{1}, t_{2})+M_{3}(t_{1}, t_{2}),
\]
where $0\leq t_{1} <t_{2} \leq T_{1}$. From corollary~\ref{oszacg12} and the definition of the space $X(T_{1})$ we have
\eqq{g(t) \leq \cmt \tgj, \hd \hd  |h'(t) | \leq \| h \|_{X(T_{1})} \tgj. }{xa}
We shall show that
\eqq{\lim_{t_{2} \rightarrow t_{1}^{+}}M_{i}(t_{1}, t_{2})=0, \hd \hd i=1,2,3. }{xb}
Applying (\ref{xa}) we have
\[
\int_{0}^{t_{2}} g(\tau) |h'(t_{2}- \tau)| \dt \leq \cmt  \| h \|_{X(T_{1})}  t_{2}^{2\gamma-1} B(\gamma, \gamma),
\]
hence $M_{1}(t_{1}, t_{2}) \leq \cmt  \| h \|_{X(T_{1})}  B(\gamma, \gamma) |1- (\frac{t_{1}}{t_{2}})^{1- \gamma}| t_{2}^{\gamma}$, thus we have (\ref{xb}) for $i=1$ and for any nonnegative~$t_{1}$.

\no Because $M_{2}(0, t_{2})=M_{3}(0, t_{2})=0 $ it is sufficient to consider (\ref{xb}) for positive $t_{1}$. In this case, after applying (\ref{xa}) we have
\[
\int_{t_{1}}^{t_{2}} g(\tau) |h'(t_{2}- \tau)| \dt \leq \cmt  \| h \|_{X(T_{1})}  t_{2}^{2\gamma-1} \int_{\frac{t_{1}}{t_{2}}}^{1} s^{\gamma-1}(1-s)^{\gamma-1}ds \longrightarrow 0,
\]
if $t_{2} \rightarrow t_{1}^{+}$ and we have (\ref{xb}) for $i=2$. Finally, if $t_{0}\in (0, t_{1})$ then we write
\[
\int_{0}^{t_{1}} g(\tau) |h'(t_{2}- \tau)-h'(t_{1}- \tau)| \dt
\]
\[
\leq 2 \sup_{s\in [t_{1}-t_{0}, t_{2}]} |h'(s)| \int_{0}^{t_{0} } g(\tau) \dt + g(t_{0}) \int_{t_{0}}^{t_{1}} | h'(t_{2}-\tau)- h'(t_{1}- \tau)|\dt
\]
\[
\leq 2   \| h \|_{X(T_{1})} (t_{1}-t_{0})^{\gamma-1}\hspace{-0.2cm} \int_{0}^{t_{0} } g(\tau) \dt + \cmt t_{0}^{\gamma-1}\hspace{-0.2cm} \int_{0}^{t_{1}-t_{0}}\hspace{-0.4cm} | h'(t_{2}-t_{1}+s)- h'(s)|ds.
\]
The first term is arbitrary small, if $t_{0}$ is sufficient small. Then for fixed $t_{0}$ the second term can be made arbitrary small, provided  $|t_{2}-t_{1}|$ is sufficient small, because the translation operator is continuous in $L^{1}$. It proves (\ref{xb}) for $i=3$ and we have $Pc \in X(T_{1})$, provided $c\in X(T_{1})$.

It remains to show that operator $P$ is a contraction on $X(T_{1})$ for $T_{1}$ small enough.
 Using the linearity of $P$ it is sufficient to show that for $c = c_{1} - c_{2},$ where $c_{1},c_{2}  \in X(T_{1})$ we have $\norm{Pc}_{X(T_{1})} \leq \alpha \norm{c}_{X(T_{1})}$, where $\al < 1.$ For this purpose we denote  $\norm{\cdot} = \norm{\cdot}_{C[0,T_{1}]}.$ Then from  (\ref{xa}) we get
\[
\norm{Pc} = \norm{g*A^{n}c} \leq \cmt \norm{A^{n}} \norm{c} \int_{0}^{T_{1}} \tau^{\gamma-1}d\tau \leq \frac{\cmt}{\gamma} \norm{A^{n}} \norm{c} T_{1}^{\gamma}.
\]
Using $c(0)=0$ we obtain
\[
\norm{t^{1-\gamma}(Pc)'}
\]
\[
 =
 \norm{t^{1-\gamma} \izt g(t-\tau){A^{n}}'(\tau)c(\tau) d\tau+t^{1-\gamma} \izt g(t-\tau)A^{n}(\tau)c'(\tau) d\tau}.
\]
From definition $X(T_{1})$ and  (\ref{xa}) we have
\[
\abs{t^{1-\gamma} \izt g(t-\tau){A^{n}}'(\tau)c(\tau) d\tau} \leq \cmt  t^{1-\gamma}\norm{{A^{n}}'} \norm{c} \izt (t-\tau)^{\gamma-1}d\tau
\]
\[
 \leq \frac{\cmt}{\gamma} \norm{{A^{n}}'} \norm{c}T_{1} .
\]
Next, we have
\[
\abs{t^{1-\gamma} \izt g(t-\tau)A^{n}(\tau)c'(\tau) d\tau}
\]
\[
 \leq \cmt t^{1-\gamma}\norm{A^{n}}\norm{c}_{X(T_{1})}\izt (t-\tau)^{\gamma-1}\tau^{\gamma-1}d\tau
 \]
 \[
 =\cmt t^{\gamma}\norm{A^{n}}\norm{c}_{X(T_{1})}B(\gamma,\gamma).
\]
Thus for $T_{1}$ such, that
\[
4 T^{1-\gamma} \frac{\cmt}{\gamma} \norm{A^{n}}_{C^{1}[0,T_{1}]} T_{1}^{\gamma} < 1
\]
$P$ is a contraction on $X(T_{1})$ and we obtained the solution of (\ref{cnc}) in $X(T_{1}).$

In order to extend the solution on the whole interval $[0,T]$, let us assume that we already defined solution $c_{s}$ to  (\ref{cnc}) on the interval $[0,T_{k}]$ for some $T_{k} > 0.$ We would like to find the solution on $[T_{k}, T_{k+1}],$ where $T_{k+1} > T_{k}.$ To that end, we define the space
\[
X_{k}(T_{k+1}) = \{ c \in C^{1}((0,T_{k+1}];\mr^{n}): c(t) = c_{s}(t) \textrm{ for } t \in [0,T_{k}]\}
\]
equipped with metrics $\norm{c}_{X_{k}(T_{k+1})} = \norm{c}_{C^{1}[T_{k},T_{k+1}]}.$ Then $X_{k}(T_{k+1})$ is complete metric space.  Let $c\in X_{k}(T_{k+1}),$ then $c\in X(T_{k+1})$ and from the previous part of the proof we get that $Pc \in X(T_{k+1}).$ Making use of the definition $X(T_{k+1})$ we obtain, that $Pc \in C[0,T_{k+1}]$ and $t^{1-\gamma}(Pc)' \in C[0,T_{k+1}],$ which implies that $(Pc)' \in C(0,T_{k+1}].$ Thus we have just shown that $Pc \in X_{k}(T_{k+1}).$

Now we will prove that operator $P$ is a contraction on $X_{k}(T_{k+1})$ for $T_{k+1} - T_{k}$ small enough. We shall rewrite $P$ in the form
\[
(Pc)(t) = c_{n,0} + \izt g(t-\tau)F_{n}(\tau)d\tau - \int_{0}^{T_{k}}g(t-\tau)A^{n}(\tau)c(\tau)d\tau
\]
\[
 - \int_{T_{k}}^{t}g(t-\tau)A^{n}(\tau)c(\tau)d\tau.
\]
Let $c = c_{1} - c_{2},$ where $c_{1}, c_{2} \in X_{k}(T_{k+1}).$ Of course $c \equiv 0$ on $[0,T_{k}]$ and
\[
(Pc)(t) = -\int_{T_{k}}^{t}g(t-\tau)A^{n}(\tau)c(\tau)d\tau
\]
Denoting by $\norm{\cdot}_{C[T_{k},T_{k+1}]} = \norm{\cdot},$ from (\ref{xa})  we have
\[
\norm{Pc} \leq \frac{\cmt}{\gamma} \norm{A^{n}}\norm{c}(T_{k+1}-T_{k})^{\gamma}.
\]
Similarly, we get
\[
\norm{(Pc)'}
\leq \frac{\cmt}{\gamma}\left[\norm{A^{n'}}\norm{c}+
\norm{A^{n}}\norm{c'}\right](T_{k+1}-T_{k})^{\gamma},
\]
thus
\[
\norm{Pc}_{X_{k}(T_{k+1})} \leq 2\frac{\cmt}{\gamma} \norm{A^{n}}_{C^{1}[0,T]} (T_{k+1}-T_{k})^{\gamma}\norm{c}_{X_{k}(T_{k+1})}.
\]
To sum up, the operator $P$ is contraction on $X_{k}(T_{k+1})$, if  $T_{k+1}$ is such, that
\[
2\frac{\cmt}{\gamma} \norm{A^{n}}_{C^{1}[0,T]} (T_{k+1}-T_{k})^{\gamma} < 1.
\]
It is clear that the condition on  the difference $|T_{k+1}-T_{k}|$ does not depend on  $k$, hence after finite number of steps we obtain that $P$ has a unique fixed point on $[0,T]$ and the lemma is proven.
\end{proof}

\begin{rem}
The approximate solution $c_{n}$ of (\ref{cnc}) is absolutely continuous, thus $c_{n}$ is also a solution to  (\ref{cn}). What is more, $u^{n}$ given by (\ref{un}) satisfy the system of equations (\ref{nn}), $ u^{n}(x,\cdot) \in AC[0,T]$ and $t^{1-\gamma}u^{n}_{t}(x,\cdot) \in C[0,T].$ Besides, for $\beta \in \mathbb{N}^{\mathbb{N}},$ if $\p \Om $ is regular enough, then $D_{x}^{\beta}u^{n}(x,\cdot) \in AC[0,T]$ and $t^{1-\gamma}D_{x}^{\beta}u^{n}_{t}(x,\cdot) \in C[0,T].$
\label{remone}
\end{rem}

\subsection{Energy estimates}

In this section we shall obtain a uniform bound for the sequence of approximate solution. First we formulate the following lemma, where we denote $\map = \mg $ .

\begin{lem}\label{Alemat}
Assume that $w \in L^{2}(\Om^{T})$ and for $x\in \Om$ $w(x,\cdot) \in AC[0,T]$ and for fixed $\gamma \in (0,1)$ \hd $t^{1-\gamma}w_{t}(x,t) \in L^{\infty}(\Om^{T}).$ Then the following equality holds
\[
\dm ||w(\cdot,t)||^{2}_{L^{2}(\Om)} +\izj\hspace{-0.1cm}  \izt (t-\tau)^{-\al - 1}\hspace{-0.2cm}\iO |w(x,t) - w(x,\tau)|^{2}dx   d\tau \al \map  d\al
\]
 \eqq{+ \izj   t^{-\al}\iO |w(x,t) - w(x,0)|^{2}dx \map d\al = 2 \iO \dm w(x,t) w(x,t)dx. }{oszacA}
\end{lem}
\begin{proof}
The proof of that lemma is almost the same as the proof of Lemma 2 in \cite{KY}, thus is omitted.
\end{proof}

\no Now we are able to formulate and prove energy estimates for approximate solutions.
\begin{lem}\label{oszac1}
Assume that $\uz \in L^{2}(\Om),$ $f\in L^{2}(0,T;H^{-1}(\Om)),$ $a_{i,j}$ are measurable and $a_{i,j}= a_{j,i},$ for some  $p_{1},p_{2} \in [2,\frac{2N}{N-2})$ $b_{j} \in L^{\infty}(0,T;L^{\frac{2p_{1}}{p_{1}-2}}(\Om)) $ for $j=1\cdots N$ and $c \in L^{\infty}(0,T; L^{\frac{p_{2}}{p_{2}-2}}(\Om)).$ What is more, assume that (\ref{elipt}) holds. Then, for every $n \in \mathbb{N}$ and for every $t \in (0,T]$ the approximate solution  $u^{n}$ given by (\ref{un}) and lemma~\ref{lemthree} satisfies the inequality
\[
\izj  I^{1-\al}\norm{u^{n}(\cdot,t)}_{L^{2}(\Om)}^{2} \ma d\al
\]
\[
 + \izt \hspace{-0.1cm}\izj \hspace{-0.1cm} \iO\abs{u^{n}(x,\tau) - u^{n}_{0}(x)}^{2}dx  \tau^{-\al} \map  d\al d\tau+ \lambda \hspace{-0.1cm} \izt\hspace{-0.2cm}\norm{D u^{n}(\cdot,\tau)}^{2}_{L^{2}(\Om)}\hspace{-0.1cm}d\tau
 \]
 \[
+  \izt\izj  \int_{0}^{\tau} (\tau-s)^{-\al-1}\iO \abs{u^{n}(x,\tau)-u^{n}(x,s)}^{2}dxds \al \map d\al d\tau
\]
 \[\leq \tilde{c}_{1} \left( \norm{u_{0}}^{2}_{L^{2}(\Om)} +\izt \norm{f(\cdot,\tau)}_{H^{-1}(\Om)}^{2} d\tau \right)+ \delta_{n},
\]
where $\tilde{c}_{1}$ depends only on  $\Om,T,\mu,p_{1},p_{2},\lambda, \norm{c^{n}}_{L^{\infty}(0,T;L^{\frac{p_{2}}{p_{2}-2}}(\Om))},$\\
$\norm{b^{n}}^{2}_{L^{\infty}(0,T;L^{\frac{2 p_{1}}{p_{1}-2}}(\Om))}$
and $\delta_{n}\overset{n\rightarrow \infty}{\rightarrow} 0$ uniformly with respect to $t$.
\end{lem}

\begin{proof}
We multiply (\ref{przyb}) by $c_{n,m}(t)$ and sum it up from $1$ to $n.$
\[
\iO \dm u^{n}(x,t)u^{n}(x,t)dx + \sum_{i,j=1}^{N}\iO a_{i,j}^{n}(x,t)D_{j}u^{n}(x,t)D_{i}u^{n}(x,t)dx
\]
\[
 = \sum_{j=1}^{N}\iO b_{j}^{n}(x,t)D_{j} u^{n}(x,t)u^{n}(x,t)dx + \iO c^{n}(x,t)\abs{u^{n}(x,t)}^{2}dx
 \]
 \[
 + \left\langle f_{\frac{1}{n}}(x,t), u^{n}(x,t)\right\rangle_{H^{-1}(\Om) \times H_{0}^{1}(\Om)}.
\]
Using (\ref{elipt}) and applying Lemma~\ref{Alemat}, we get
\[
\dm \norm{u^{n}(\cdot,t)}_{L^{2}(\Om)}^{2} +2 \lambda \norm{D u^{n}(\cdot,t)}^{2}_{L^{2}(\Om)}
\]
\[
 + \izj  t^{-\al}\iO\abs{u^{n}(x,t) - u^{n}_{0}(x)}^{2}dx \map d\al
\]
\[
+  \izj \int_{0}^{t} (t-\tau)^{-\al-1}\iO \abs{u^{n}(x,t)-u^{n}(x,\tau)}^{2}dxd\tau \al \map  d\al
\]
\[
\leq 2\sum_{j=1}^{N}\iO b_{j}^{n}(x,t)D_{j} u^{n}(x,t)u^{n}(x,t)dx
\]
\eqq{
 + 2\iO c^{n}(x,t)\abs{u^{n}(x,t)}^{2}dx+ 2\left\langle f_{\frac{1}{n}}(x,t), u^{n}(x,t)\right\rangle_{H^{-1}(\Om) \times H_{0}^{1}(\Om)}.
}{ha}
Denoting by $b^{n} = (b_{1}^{n},\cdots b_{N}^{n})$
we can estimate the right hand side as follows.
\[
RHS \leq 2\norm{Du^{n}(\cdot,t)}_{L^{2}(\Om)}
\norm{u^{n}(\cdot,t)}_{L^{p_{1}}(\Om)}
\norm{b^{n}(\cdot,t)}_{L^{\frac{2p_{1}}{p_{1}-2}}(\Om)} +
\]
\[
2\norm{c^{n}(\cdot,t)}_{_{L^{\frac{p_{2}}{p_{2}-2}}(\Om)}}\hspace{-0.3cm}
\norm{u^{n}(\cdot,t)}^{2}_{_{L^{p_{2}}(\Om)}}\hspace{-0.2cm} + \frac{\lambda}{4}\norm{Du^{n}(\cdot,t)}^{2}_{_{L^{2}(\Om)}} + \frac{4}{\lambda}\norm{f_{\frac{1}{n}}(\cdot,t)}_{_{H^{-1}(\Om)}}^{2}
\]
\[
\leq \hspace{-0.1cm} \frac{\lambda}{2}\hspace{-0.1cm}\norm{Du^{n}(\cdot,t)}^{2}_{_{L^{2}(\Om)}}\hspace{-0.3cm} + \frac{4}{\lambda}\norm{b^{n}(\cdot,t)}^{2}_{_{L^{\frac{2p_{1}}{p_{1}-2}}(\Om)}}\hspace{-0.5cm}\left[\ve_{1}\norm{Du^{n}(\cdot,t)}^{2}_{_{L^{2}(\Om)}}
\hspace{-0.4cm}+\eta_{1}\hspace{-0.1cm}\norm{u^{n}(\cdot,t)}^{2}_{_{L^{2}(\Om)}}\hspace{-0.1cm} \right]
\]
\[
+2\norm{c^{n}(\cdot,t)}_{L^{\frac{p_{2}}{p_{2}-2}}(\Om)}\left(\ve_{2}\norm{Du^{n}(\cdot,t)}^{2}_{L^{2}(\Om)}+\eta_{2}
\norm{u^{n}(\cdot,t)}^{2}_{L^{2}(\Om)}\right)
\]
\[
+ \frac{4}{\lambda}\norm{f_{\frac{1}{n}}(\cdot,t)}_{H^{-1}(\Om)}^{2},
\]
where $\eta_{i}$ depends on $\ve_{i},p_{i}$ for $i=1,2.$ If we take $\ve_{1}, \ve_{2}$ small enough and use this inequality in (\ref{ha}) we obtain that
\[
\dm \norm{u^{n}(\cdot,t)}_{L^{2}(\Om)}^{2} + \lambda \norm{Du^{n}(\cdot,t)}^{2}_{L^{2}(\Om)}
\]
\[
 + \izj  t^{-\al}\iO\abs{u^{n}(x,t) - u^{n}_{0}(x)}^{2}dx \map d\al
\]
\[
+  \izj  \int_{0}^{t} (t-\tau)^{-\al-1}\iO \abs{u^{n}(x,t)-u^{n}(x,\tau)}^{2}dxd\tau \al \map d\al
\]
\eqq{
 \leq h_{n}(t)\norm{u^{n}(\cdot,t)}^{2}_{L^{2}(\Om)} + \frac{4}{ \lambda}\norm{f_{\frac{1}{n}}(\cdot,t)}_{H^{-1}(\Om)}^{2},
}{hb}
where $h_{n}(t)$ is a function which depends on constants $p_{1},p_{2},\lambda$ and norms $\norm{c^{n}(\cdot,t)}_{L^{\frac{p_{2}}{p_{2}-2}}(\Om)}$, $\norm{b^{n}(\cdot,t)}^{2}_{L^{\frac{2 p_{1}}{p_{1}-2}(\Om)}}$.
Passing over all the rest terms on the left hand side apart from the first one and applying to both sides of the inequality operator $\im$ we obtain
\[
\norm{u^{n}(\cdot,t)}_{L^{2}(\Om)}^{2} \leq \norm{u^{n}(\cdot,0)}_{L^{2}(\Om)}^{2} + \tilde{h}_{n}(t) \cdot ( g*\norm{u^{n}(\cdot,t)}^{2}_{L^{2}(\Om)})
\]
\[
  + \frac{4}{ \lambda}g*\norm{f_{\frac{1}{n}}(\cdot,t)}_{H^{-1}(\Om)}^{2},
\]
where $\tilde{h}_{n} = \norm{h_{n}}_{L^{\infty}(0,t)}.$
We apply Lemma \ref{gronwall} with function
\[
a(t) = \norm{u^{n}(\cdot,0)}_{L^{2}(\Om)}^{2} + \frac{4}{ \lambda}g*\norm{f_{\frac{1}{n}}(\cdot,t)}_{H^{-1}(\Om)}^{2}
\]
and we get that
\eqq{\norm{u^{n}(\cdot,t)}_{_{L^{2}(\Om)}}^{2}\hspace{-0.2cm} \leq \sum_{k=0}^{\infty}\tilde{h}_{n}^{k}(t)\left[g^{[k]}*\hspace{-0.1cm}\left[
\norm{u^{n}(\cdot,0)}_{_{L^{2}(\Om)}}^{2}\hspace{-0.2cm} +\frac{4}{ \lambda}g*\hspace{-0.1cm}\norm{f_{\frac{1}{n}}(\cdot,t)}_{_{H^{-1}(\Om)}}^{2}\right] \right]}{hc}
From Lemma~\ref{gronwall} we can also conclude that the series above is uniformly convergent.

\no We come back to inequality (\ref{hb}) and integrate it from $0$ to $t.$
\[
\izj  I^{1-\al}\norm{u^{n}(\cdot,t)}_{L^{2}(\Om)}^{2} \ma d\al + \lambda \izt\norm{D u^{n}(\cdot,\tau)}^{2}_{L^{2}(\Om)}d\tau
\]
\[
 +\izt\izj  \tau^{-\al}\iO \abs{u^{n}(x,\tau) - u^{n}_{0}(x)}^{2}dx \map  d\al d\tau
\]
\[
+ \izt\izj  \int_{0}^{\tau} (\tau-s)^{-\al-1}\iO \abs{u^{n}(x,\tau)-u^{n}(x,s)}^{2}dxds \al \map d\al d\tau
\]
\[
\leq \izj  I^{1-\al}\norm{u^{n}_{0}}^{2}_{L^{2}(\Om)}\ma  d\al+ \izt h_{n}(\tau)\norm{u^{n}(\cdot,\tau)}^{2}_{L^{2}(\Om)}d\tau
\]
\[
 + \izt\frac{4}{ \lambda}\norm{f_{\frac{1}{n}}(\cdot,\tau)}_{H^{-1}(\Om)}^{2}d\tau.
\]
From (\ref{hc}) and (\ref{groncn}) we conclude that
\[
\izt\hspace{-0.2cm} h_{n}(\tau)\norm{u^{n}(\cdot,\tau)}^{2}_{L^{2}(\Om)}d\tau \leq \hspace{-0.2cm}\izt \hspace{-0.2cm} h_{n}(\tau) \sum_{k=0}^{\infty}\tilde{h}_{n}^{k}(\tau)\left[g^{[k]}*
\norm{u^{n}(\cdot,0)}_{L^{2}(\Om)}^{2}\right]d\tau
\]
\[
 +  \frac{4}{\lambda}\izt h_{n}(\tau) \sum_{k=0}^{\infty}\tilde{h}_{n}^{k}(\tau)\left[g^{[k+1]}
 *\norm{f_{\frac{1}{n}}(\cdot,\tau)}_{H^{-1}(\Om)}^{2}\right]
\]
\[
\leq \tilde{h}_{n}(t)\izt \sum_{k=0}^{\infty} \tilde{h}_{n}^{k}(\tau) \cmt^{k} \ggg^{k} I^{k\gamma}  \norm{u^{n}(\cdot,0)}_{L^{2}(\Om)}^{2} \dt
\]
\[
+\frac{4}{\lambda} \tilde{h}_{n}(t) \izt  \sum_{k=0}^{\infty} \tilde{h}_{n}^{k}(\tau) \cmt^{k+1} \ggg^{k+1} I^{(k+1)\gamma} \norm{f_{\frac{1}{n}}(\cdot,\tau)}_{H^{-1}(\Om)}^{2}d\tau
\]
\[
\leq \tilde{h}_{n}(t) \sum_{k=0}^{\infty} \tilde{h}_{n}^{k}(t) \izt  \cmt^{k} \ggg^{k} I^{k\gamma}  \norm{u^{n}(\cdot,0)}_{L^{2}(\Om)}^{2} \dt
\]
\[
+\frac{4}{\lambda}  \sum_{k=0}^{\infty} \left[ \cmt \ggg   \tilde{h}_{n}(t)  \right]^{k+1} I^{(k+1)\gamma+1} \norm{f_{\frac{1}{n}}(\cdot,t)}_{H^{-1}(\Om)}^{2}
\]
\[
\leq \tilde{h}_{n}(t) \norm{u^{n}(\cdot,0)}_{L^{2}(\Om)}^{2}
\sum_{k=0}^{\infty} \tilde{h}_{n}^{k}(t)   \cmt^{k} \ggg^{k} \frac{t^{k\gamma+1}}{\Gamma(k\gamma+2)}
\]
\[
+\frac{4}{\lambda}  \sum_{k=0}^{\infty} \left[ \cmt \ggg   \tilde{h}_{n}(t)  \right]^{k+1} \frac{t^{(k+1)\gamma}}{\Gamma((k+1)\gamma+1)} \izt  \norm{f_{\frac{1}{n}}(\cdot,\tau)}_{H^{-1}(\Om)}^{2} \dt
\]
\[
\leq \tilde{h}_{n}(t) tE_{\gamma, 2}\left( \tilde{h}_{n}(t)   \cmt \ggg^{k} t^{\gamma} \right) \norm{u^{n}(\cdot,0)}_{L^{2}(\Om)}^{2}
\]
\[
+\frac{4}{\lambda}   E_{\gamma, 1} \left(
\cmt \ggg   \tilde{h}_{n}(t) t^{\gamma}\right) \izt \norm{f_{\frac{1}{n}}(\cdot,\tau)}_{H^{-1}(\Om)}^{2} \dt,
\]
\[
\leq c_{1} \left( \norm{u^{n}(\cdot,0)}_{L^{2}(\Om)}^{2} +   \izt \norm{f_{\frac{1}{n}}(\cdot,\tau)}_{H^{-1}(\Om)}^{2} \dt \right) ,
\]
where
\[
c_{1}= \tilde{h}_{n}(T) TE_{\gamma, 2}\left( \tilde{h}_{n}(T)   \cmt \ggg^{k} T^{\gamma} \right) +\frac{4}{\lambda}   E_{\gamma, 1} \left(
\cmt \ggg   \tilde{h}_{n}(T) T^{\gamma}\right)
\]
 and  $E_{\gamma, \beta}$ denotes Mittag-Leffler function, i.e. $E_{\gamma, \beta}(x)= \sum_{k=0}^{\infty} \frac{x^{k}}{\Gamma(k\gamma+ \beta)}$. Finally, we notice that
\[
\izt \norm{f_{\frac{1}{n}}(\cdot,\tau)}_{H^{-1}(\Om)}^{2} d\tau \leq \hspace{-0.2cm} \izt \hspace{-0.2cm} \norm{f(\cdot,\tau)}_{H^{-1}(\Om)}^{2} d\tau + \int_{t}^{t+\frac{1}{n}}\hspace{-0.2cm}\norm{f(\cdot,\tau)}_{H^{-1}(\Om)}^{2} d\tau.
\]
Denoting by
\[
\delta_{n} = \sup_{t \in [0,T-\frac{1}{n})} \int_{t}^{t+\frac{1}{n}}\norm{f(\cdot,\tau)}_{H^{-1}(\Om)}^{2} d\tau
\]
and using assumptions concerning $f$ we obtain that $\delta_{n}\rightarrow 0$ as $n\rightarrow \infty$ uniformly with respect to $t.$
The application of the Bessel inequality $\norm{\uz^{n}}_{L^{2}(\Om)} \leq \norm{\uz}_{L^{2}(\Om)}$  and the estimate
\[
\izj  I^{1-\al}\norm{u^{n}_{0}}^{2}_{L^{2}(\Om)} \ma  d\al \leq \norm{\uz}_{L^{2}(\Om)}^{2} \izj \izt (t-\tau)^{-\al} \dt \map d\al
\]
\[
 = \norm{\uz}_{L^{2}(\Om)}^{2}\izj \frac{t^{1-\al}}{\Gamma(2-\al)} \ma d\al \leq 2 \max\{1,T\}c_{\mu} \norm{\uz}_{L^{2}(\Om)}^{2}
\]
finishes the proof.
\end{proof}

\subsection{Limit passage}
\label{limpass}

From Lemma \ref{oszac1} follows that the sequence $\left\{\nabla u^{n}\right\}$ is bounded in $L^{2}(0,T;L^{2}(\Om)).$ Thus we can choose subsequence $\left\{ u^{n}\right\}$ (still indexed by $n$) such, that
\eqq{u^{n}\rightharpoonup u \textrm { in } L^{2}(0,T;H^{1}_{0}(\Om)).}{zb1}
Denote by $\norm{\cdot}$ the norm $\norm{\cdot}_{L^{2}(\Om)}$. We notice that
\[
\norm{I^{1-\al}u^{n}(\cdot,t)} = \sup_{\norm{h}=1} \iO \abs{h(x)\frac{1}{\Gamma(1-\al)}\izt (t-\tau)^{-\al}u^{n}(x,\tau)d\tau}dx
\]
\[
 \leq \sup_{\norm{h}=1} \izt \frac{(t-\tau)^{-\al}}{\Gamma(1-\al)} \iO \abs{h(x)u^{n}(x,\tau)}dxd\tau
\]
\[
\leq \izt \frac{(t-\tau)^{-\al}}{2\Gamma(1-\al)}\norm{u^{n}(\cdot,\tau)}^{2}d\tau + \izt \frac{(t-\tau)^{-\al}}{2\Gamma(1-\al)}d\tau
\]
\[
 = \frac{1}{2}I^{1-\al}\norm{u^{n}(\cdot,t)}^{2} + \frac{T^{1-\al}}{2\Gamma(2-\al)}.
\]
Using this calculations, we have
\[
\norm{\izj  I^{1-\al}u^{n}(\cdot,t) \ma d\al} \leq \izj  \norm{I^{1-\al}u^{n}(\cdot,t)} \ma d\al
\]
\[
 \leq \frac{1}{2}\izj I^{1-\al}\norm{u^{n}(\cdot,t)}^{2} \ma  d\al +\frac{1}{2} \izj T^{1-\al} \frac{\ma}{\Gamma(2-\al)}d\al ,
\]
and the last term is bounded by $\max\{T,1\} c_{\mu}$.   From Lemma \ref{oszac1} we get the bound for the sequence $\left\{\izj  I^{1-\al}u^{n}(\cdot,t) \ma  d\al\right\}$ in $L^{\infty}(0,T;L^{2}(\Om)).$\\
We will find the estimate for $\dm u^{n}.$ To that end, we choose arbitrary $w \in H_{0}^{1}(\Om).$ Then there exist coefficients $d_{m}$ such, that $w(x) = \sum_{m=1}^{\infty}d_{m}\vf_{m}(x).$ We multiply (\ref{przyb}) by $d_{m}$ and sum it up from 1 to $n.$ Denoting $w^{n} = \sum_{m=1}^{n}d_{m}\vf_{m}(x)$ we get
\[
\iO \dm u^{n}(x,t)w^{n}(x)dx = -\sum_{i,j=1}^{N}a^{n}_{i,j}(x,t)D_{j}u^{n}(x,t)D_{i}w^{n}(x)dx
\]
\[
+\sum_{j=1}^{N}\iO b_{j}^{n}(x,t)D_{j}u^{n}(x,t)w^{n}(x)dx + \iO c^{n}(x,t)u^{n}(x,t)w^{n}(x)dx
\]
\[
+ \left\langle f_{\frac{1}{n}}(x,t),w^{n}(x)\right\rangle_{H^{-1}(\Om) \times H_{0}^{1}(\Om)}.
\]
We note that $\iO \dm u^{n}(x,t)w(x)dx= \iO \dm u^{n}(x,t)w^{n}(x)dx$.
Using (\ref{elipt}) and H\"older inequality we obtain
\[
\abs{\iO \dm u^{n}(x,t)w(x)dx} \leq \Lambda \norm{D u^{n}(\cdot,t)}_{L^{2}(\Om)}\norm{D w^{n}}_{L^{2}(\Om)}
\]
\[
 + \norm{b}_{L^{\frac{2p_{1}}{p_{1}-2}}(\Om)}\norm{Du^{n}}_{L^{2}(\Om)}\norm{w^{n}}_{L^{p_{1}}(\Om)}\]
\[
 +\norm{c}_{L^{\frac{p_{2}}{p_{2}-2}}(\Om)}\norm{u^{n}}_{L^{p_{2}}(\Om)}\norm{w^{n}}_{L^{p_{2}}(\Om)}+ \norm{f_{\frac{1}{n}}}_{H^{-1}(\Om)}\norm{w^{n}}_{H_{0}^{1}(\Om)}.
\]
Taking advantage from the last inequality, the estimate from Lemma \ref{oszac1}, making use of Sobolev embedding and Poincare inequality we conclude that
\[
\norm{D^{(\mu)}u^{n}}_{L^{2}(0,T;H^{-1}(\Om))} = \norm{\poch \izj \ma I^{1-\al}[u^{n}-\uz^{n}]d\al}_{L^{2}(0,T;H^{-1}(\Om))}
\]
\eqq{
 \leq c_{2} \left( \norm{u_{0}}^{2}_{L^{2}(\Om)} +\izt \norm{f(\cdot,\tau)}_{H^{-1}(\Om)}^{2} d\tau \right)+ \delta_{n},
}{dmszac}
where $c_{2} ={c}_{2}\left( \Lambda, \widetilde{c}_{1} \right)$ and $\widetilde{c}_{1}$, $\delta_{n}$ are from lemma~\ref{oszac1}. Thus the sequence
$\{ \izj I^{1-\al}u^{n}(\cdot,t) \ma d\al \}$  is bounded in $_{0}H^{1}(0,T;H^{-1}(\Om))$ and then exists a subsequence (still indexed by $n$) such, that
\eqq{\izj  I^{1-\al}[u^{n}-\uz^{n}]  \ma d\al\rightharpoonup v \textrm{ in } _{0}H^{1}(0,T;H^{-1}(\Om)).}{zb2}
We will show that $\poch v = \poch \izj \ma I^{1-\al}[u-\uz] d\al$ in a weak sense. Assume that $\Phi \in C_{0}^{\infty}(0,T)$ and $\vf \in H_{0}^{1}(\Om).$ Then
\[
\int_{0}^{T} \Phi(t)\left\langle \poch v(\cdot,t),\vf\right\rangle_{H^{-1}(\Om) \times H_{0}^{1}(\Om)}dt
\]
\[
 = \lim_{n\rightarrow \infty} \int_{0}^{T} \Phi(t) \iO\poch \izj  I^{1-\al}[u^{n}-\uz^{n}] \ma d\al\vf(x)dxdt
\]
\[
 =- \lim_{n\rightarrow \infty}\iO \int_{0}^{T} \Phi'(t)\izj  I^{1-\al}[u^{n}-\uz^{n}]\ma d\al dt\vf(x)dx
\]
\[
=- \lim_{n\rightarrow \infty}\iO \int_{0}^{T} \Phi'(t)\izj  I^{1-\al}[u-\uz] \ma d\al dt\vf(x)dx,
\]
where the last equality is implied by the fact that $I^{1-\al}$ is continuous in  $L^{2}(0,T),$ so it is also weakly continuous. Thus  $\poch v = \poch\izj  I^{1-\al}[u-\uz]\ma  d\al $ in a weak sense, so on the subsequence we have the following weak convergence in $H_{0}^{1}(0,T;H^{-1}(\Om))$
\eqq{\izj  I^{1-\al}[u^{n}(x,t)-\uz^{n}(x)] \ma d\al \rightharpoonup \izj  I^{1-\al}[u(x,t)-\uz(x)] \ma d\al.}{zb3}
Having the above convergence we can obtain (\ref{slabadef}).
By density argument it is enough to show  (\ref{slabadef}) for $w(x) = \sum_{m=1}^{K}d_{m}\vf_{m}(x),$ where $d_{m}$ are some arbitrary constants. We multiply (\ref{przyb}) for selected subsequence by $d_{m}$ and sum it up from 1 to $K$. Next we fix $t_{0}\in (0,T)$ and multiply the equation by $\eta_{\ve}(t+t_{0})$, where $\eta_{\ve}$ stands for standard smoothing kernel. Integrating obtained equality from zero to $T,$ we get
\[
\int_{0}^{T}\eta_{\ve}(t+t_{0})\iO\izj\poch I^{1-\al}[u^{n}(x,t)-\uz^{n}(x)] \ma d\al w(x)dx
\]
\[
 + \sum_{j,j=1}^{N}\int_{0}^{T}\eta_{\ve}(t+t_{0})\iO a_{i,j}^{n}(x,t)D_{j}u^{n}(x,t)D_{i}w(x)dxdt
 \]
 \[
= \sum_{j=1}^{N}\int_{0}^{T}\iO b_{j}^{n}(x,t)D_{j}u^{n}(x,t)D_{i}w(x)\eta_{\ve}(t+t_{0})dxdt
 \]
\[
+\int_{0}^{T}\hspace{-0.2cm}\left(\iO c^{n}(x,t)u^{n}(x,t)D_{i}w(x)dx+
\iO f_{\frac{1}{n}}(x,t)w(x)dx\right)\eta_{\ve}(t+t_{0})dt.
\]
Firstly, we pass to the limit with $n$, and then with $\ve.$ For $\ve < T -t_{0}$, using (\ref{zb3}) we have
\[
\int_{0}^{T}\eta_{\ve}(t+t_{0})\iO\poch\izj I^{1-\al}[u^{n}(x,t)-\uz^{n}(x)]\ma  d\al w(x)dx
\]
\[
= - \int_{0}^{T}\eta'_{\ve}(t+t_{0})\iO\izj I^{1-\al}[u^{n}(x,t)-\uz^{n}(x)] \ma d\al w(x)dx
\]
\[
\underset{n\rightarrow \infty}{\longrightarrow} - \int_{0}^{T}\eta'_{\ve}(t+t_{0})\iO\izj  I^{1-\al}[u(x,t)-\uz(x)] \ma d\al w(x)dx
\]
\[
=\int_{0}^{T}\eta_{\ve}(t+t_{0})\poch  \iO \izj I^{1-\al}[u(x,t)-\uz(x)] \ma  d\al w(x)dx
\]
\[
\underset{\ve\rightarrow 0}{\longrightarrow} \poch \iO \izj I^{1-\al}[u(x,t_{0})-\uz(x)] \ma d\al w(x)dx \hd \textrm{ for } a.a. \hd t_{0} \in (0,T).
\]
We proceed similarly with remaining terms. We see that $\eta_{\ve}(t+t_{0})D_{i}w(x)$ is smooth in $\Om^{T}$, thus having in mind that $a_{i,j}^{n}(x,t)\rightarrow a_{i,j}(x,t)$ in $L^{2}(\Om^{T})$ and  $D_{j}u^{n}(x,t)\rightharpoonup D_{j}u(x,t)$ in $L^{2}(\Om^{T})$ when $n\rightarrow \infty,$ we get
\[
\int_{0}^{T}\eta_{\ve}(t+t_{0})\iO a_{i,j}^{n}(x,t)D_{j}u^{n}(x,t)D_{i}w(x)dxdt
\]
\[
\underset{n\rightarrow \infty}{\longrightarrow}
 \int_{0}^{T}\eta_{\ve}(t+t_{0})\iO a_{i,j}(x,t)D_{j}u(x,t)D_{i}w(x)dxdt
\]
\[
\underset{\ve\rightarrow 0}{\longrightarrow} \iO a_{i,j}(x,t_{0})D_{j}u(x,t_{0})D_{i}w(x)dx \hd \textrm{ for } \hd  a.a. \hd t_{0} \in (0,T).
\]
From assumptions $b_{j} \in L^{\infty}(0,T;L^{2}(\Om)),$ thus $b_{j}^{n}\rightarrow b_{j}$ in $L^{2}(\Om^{T})$ and we can pass to the limit
\[
\int_{0}^{T}\eta_{\ve}(t+t_{0})\iO b^{n}_{j}(x,t)D_{j}u^{n}(x,t)w(x)dxdt
\]
\[
\underset{n\rightarrow \infty}{\longrightarrow} \int_{0}^{T}\eta_{\ve}(t+t_{0})\iO b_{j}(x,t)D_{j}u(x,t)w(x)dxdt
\]
\[
\underset{\ve\rightarrow 0}{\longrightarrow}\iO b_{j}(x,t_{0})D_{j}u(x,t_{0})w(x)dx \hd \textrm{ for } \hd a.a. \hd  t_{0} \in (0,T).
\]
We see that regularity of $c$ can be weaker then other coefficients, so we have to consider two cases. If $p_{2} \in [2,4],$ then $c\in L^{2}(\Om)$ and we can pass to the limit as in former integral. For $p_{2} > 4$ we use interpolation inequality
\[
\norm{u^{n}}_{L^{2}(0,T;L^{\frac{p_{2}}{2}}(\Om))} \leq C \norm{Du^{n}}_{L^{2}(0,T;L^{2}(\Om))}^{\theta}\norm{u^{n}}_{L^{2}(0,T;L^{2}(\Om))}^{1-\theta}.
\]
Thus $u^{n}$ is bounded in $L^{2}(0,T;L^{\frac{p_{2}}{2}}(\Om))$ - dual to $L^{2}(0,T;L^{\frac{p_{2}}{p_{2}-2}}(\Om))$ and again we can pass to the limit.
It remains to pass to the limit in the last term. We have
\[
\int_{0}^{T}\eta_{\ve}(t+t_{0})\iO f_{\frac{1}{n}}(x,t)w(x)dxdt
\]
\[
 \underset{n\rightarrow \infty}{\longrightarrow} \int_{0}^{T}\eta_{\ve}(t+t_{0}) \left\langle f(x,t)w(x)\right\rangle_{H^{-1}(\Om) \times H_{0}^{1}(\Om)} dt
\]
\[
 \underset{\ve\rightarrow 0}{\longrightarrow} \left\langle f(x,t_{0})w(x)\right\rangle_{H^{-1}(\Om) \times H_{0}^{1}(\Om)} \hd \textrm{ for } \hd a.a. \hd t_{0} \in (0,T)
\]
and that way we obtained (\ref{slabadef}) for any $\vp (x)=\sum_{m=1}^{K}d_{m} \vp_{m} $ and a.a. $t\in (0,T)$. Applying density argument we obtain (\ref{slabadef}) for all $\vp\in \hjzo $ and a.a. $t\in (0,T)$.

\subsection{Uniqueness}
 In order to prove uniqueness of the solution we need to establish the following lemma.
\begin{lem}\label{poprawa2}
Assume that  $w\in AC[0,t] $  and  $t \leq 1$. Then  the following inequality holds
\eqq{
\izj  \izt\frac{d}{d\tau} I^{1-\al}w(\tau) \cdot w(\tau)d\tau \ma d\al \geq \frac{c_{\mu}(1-\gamma)}{4\ggg} t^{-\gamma}\izt \abs{w(\tau)}^{2}d\tau.
}{ay}
\end{lem}
\begin{proof}
From proposition 10 in \cite{KY} we obtain, that if $w \in AC[0,t],$ then
\[
\izt\frac{d}{d\tau} I^{1-\al}w(\tau) \cdot w(\tau)d\tau \geq \frac{t^{-\al}}{2\Gamma(1-\al)}\izt \abs{w(\tau)}^{2}d\tau.
\]
Multiplying both sides of this inequality by $\ma$ and integrating with respect to $\al$ on $(0,1)$ we have
\[
\izj \hspace{-0.2cm} \izt\frac{d}{d\tau} I^{1-\al}w(\tau) \cdot w(\tau)d\tau \ma d\al \geq \hspace{-0.1cm} \izj \hspace{-0.2cm}  \frac{t^{-\al}}{2\Gamma(1-\al)}\izt \abs{w(\tau)}^{2}d\tau \ma  d\al.
\]
Using (\ref{defg}) we can estimate the right hand side as follows
\[
\izj  \frac{t^{-\al}}{2\Gamma(1-\al)}\izt \abs{w(\tau)}^{2}d\tau \ma  d\al \geq \frac{1}{2} \izt \abs{w(\tau)}^{2}d\tau \hspace{-0.2cm} \int_{\gamma}^{1-\gamma}\hspace{-0.1cm} \frac{t^{-\al}\ma}{\Gamma(1-\al)}  d\al
\]
\[
 \geq \frac{1}{2} \izt \abs{w(\tau)}^{2}d\tau \frac{t^{-\gamma}}{\ggg}\int_{\gamma}^{1-\gamma}\ma d\al = \frac{c_{\mu}(1-\gamma)}{4\ggg} t^{-\gamma}\izt \abs{w(\tau)}^{2}d\tau.
\]
\end{proof}

\no Now  we  prove uniqueness. Assume that $u \in L^{2}(0,T;H_{0}^{1}(\Om))$ is such that
\eqq{
\izj I^{1-\al}u(x,t) \hspace{0.1cm}  \ma d\al \in  {}_{0}H^{1}(0,T;H^{-1}(\Om))
}{xd}
satisfies (\ref{slabadef}) with  $\uz \equiv 0$ and $f \equiv 0.$ Then, for a.a. $t \in (0,T),$ we have
\[
\left\langle \poch \izj \hspace{-0.2cm}\ma I^{1-\al}u(x,t)d\al,\vf(x)
\right\rangle + \sum_{i,j=1}^{N}\iO\hspace{-0.1cm} a_{i,j}(x,t)D_{j}u(x,t)D_{i}\vf(x)dx
\]
\eqq{
=\sum_{j=1}^{N}\iO b_{j}(x,t)D_{j}u(x,t)\vf(x)dx + \iO c(x,t)u(x,t)\vf(x)dx,
}{poprawa1}
where the brackets $\langle \cdot , \cdot  \rangle$ denotes the duality pairing of $H^{1}_{0}(\Omega)$ and $H^{-1}(\Omega)$.
Let us denote $u_{n}(x,\tau) = \sum_{k=1}^{n}d_{k}(\tau)\vf_{k}(x),$ where $d_{k}(\tau) = \iO u(x,\tau)\vf_{k}(x)dx.$ We set $\vf = \vf_{k}$ in (\ref{poprawa1}), multiply it by $d_{k}$ and sum it up from $k=1$ to $n.$
\[
\izt \left\langle \poch \izj  I^{1-\al}u(x,\tau ) \ma d\al,u_{n}(x,\tau)\right\rangle \dt
\]
\[
  + \sum_{i,j=1}^{N} \izt \iO a_{i,j}(x,\tau)D_{j}u(x,\tau)D_{i}u_{n}(x,\tau)dx \dt
\]
\[
=\sum_{j=1}^{N} \izt \hspace{-0.2cm}\iO b_{j}(x,\tau)D_{j}u(x,\tau)u_{n}(x,\tau)dx \dt  + \izt\hspace{-0.2cm} \iO c(x,\tau)|u_{n}(x,\tau)|^{2}dx \dt .
\]
Making use of convergence $u_{n}\longrightarrow u$ in $L^{2}(0,t;H_{0}^{1}(\Om))$ we may write
\[
\lim_{n \rightarrow \infty}\izt \left\langle \poch \izj  I^{1-\al}u(x,\tau ) \ma d\al,u_{n}(x,\tau)\right\rangle \dt
\]
\[
  + \sum_{i,j=1}^{N} \izt \iO a_{i,j}(x,\tau)D_{j}u(x,\tau)D_{i}u(x,\tau)dx \dt
\]
\eqq{
=\sum_{j=1}^{N} \izt \iO b_{j}(x,\tau)D_{j}u(x,\tau)u(x,\tau)dx \dt  + \izt \iO c(x,\tau)|u(x,\tau)|^{2}dx \dt .
}{poprawa3}
Due to orthogonality of $\vf_{k}$ we observe that (see formula (56) in \cite{KY} for details)
\[
\izt \left\langle \frac{d}{d\tau} \izj  I^{1-\al}u(x,\tau ) \ma d\al,u_{n}(x,\tau)\right\rangle \dt
\]
\[
= \izt \left\langle \frac{d}{d\tau} \izj  I^{1-\al}u_{n}(x,\tau ) \ma d\al,u_{n}(x,\tau)\right\rangle \dt.
\]
Applying lemma~\ref{poprawa2} we obtain that
\[
\lim_{n \rightarrow \infty}\izt \left\langle \frac{d}{d\tau} \izj  I^{1-\al}u(x,\tau ) \ma d\al,u_{n}(x,\tau)\right\rangle \dt
\]
\[
= \lim_{n \rightarrow \infty}\iO \izt \frac{d}{d\tau} \izj  I^{1-\al}u_{n}(x,\tau ) \ma d\al \cdot u_{n}(x,\tau) \dt dx
\]
\[
\geq \frac{c_{\mu} (1-\gamma)}{4\Gamma(\gamma)}t^{-\gamma} \liminf_{n \rightarrow \infty} \iO \izt \abs{u_{n}(x,\tau)}^{2} d\tau dx
\]
\[
 \geq \frac{c_{\mu} (1-\gamma)}{4\Gamma(\gamma)}t^{-\gamma}  \iO \izt \abs{u(x,\tau)}^{2} d\tau dx.
\]
Finally, using ellipticity condition in (\ref{poprawa3}) and estimating the right hand side as in the proof of lemma~4 we obtain that
\[
\frac{c_{\mu}(1-\gamma)}{4\ggg} t^{-\gamma}\iO\izt\abs{u(x,\tau)}^{2}d\tau dx +\frac{\lambda}{2} \izt \norm{Du}^{2}_{L^{2}(\Om)}d\tau
\]
\[
\leq c_{0} \izt \norm{u(\cdot,\tau)}^{2}_{L^{2}(\Om)}d\tau,
\]
where $c_{0}$ depends on $\lambda$ and norms $b$ in $L^{\infty}(0,T;L^{\frac{2p_{1}}{p_{1}-2}}(\Om))$, $c$ in $L^{\infty}(0,T;L^{\frac{p_{2}}{p_{2}-2}}(\Om)).$
From this inequality we  deduce that $u \equiv  0$ on $\Om \times (0,t)$ for $t$ small enough. Repeating this argument we obtain that $u \equiv 0$ on $\Om \times (0,T),$ which  proves the uniqueness of weak solution.

\subsection{Continuity at zero}
In this section we will show that under assumption
\eqq{
\int_{\frac{1}{2}}^{1} \ma d\al > 0
}{mujd}
we have  $u \in C([0,T];H^{-1}(\Om))$ and $u|_{t=0} = \uz.$\\
We have already obtained that
\[
\izj I^{1-\al}[u(x,t) - \uz(x)] \ma d\al = k*[u(x,\cdot)- u_{0}(x)](t)
\]
is absolutely continuous function with values in $H^{-1}(\Om)$, where $k$ was defined in (\ref{k}). We also have $k*[u(x,\cdot)- u_{0}(x)](0)=0$, thus we have
\eqq{
g * \poch \left( k* [u(x,\cdot)- u_{0}(x)](t) \right) = \poch (g*k*(u-\uz)(t)) = u(\cdot,t) - \uz,
}{uuz}
because from proposition~\ref{propone} we have $g*k=1$.

From estimates we only have that $ \ddt \left(k*[u(x,\cdot)- u_{0}(x)](t) \right) $ is in $ L^{2}(0,T;H^{-1}(\Om))$ and from lemma~\ref{oszacg1} we conclude that $g \in L^{1}(0,T),$ thus in general, it is not enough to deduce that the convolution of this functions is continuous. However, under assumption (\ref{mujd}) we are able to obtain, that the kernel $g$ belongs to  $L^{2}(0,T).$ From assumption (\ref{mujd}), reasoning similarly as in the proof of (\ref{defg}),  we see that there exists    $\gamma_{0} \in \left(0,\frac{1}{4}\right)$ such that
\eqq{
 b_{\mu} \equiv \int_{\frac{1}{2}+\gamma_{0}}^{1-\gamma_{0}} \ma d\al  > 0
}{gamma1}
and we may estimate $g$ as follows
\[
g(t) \leq \frac{1}{\pi}\izi e^{-rt}\frac{1}{\izj \sin (\pi\al) r^{\al}\ma d\al}dr.
\]
Using (\ref{gamma1}) we see that for $r\leq 1$
\[
\izj \sin (\pi\al) r^{\al}\ma d\al \geq \int_{\frac{1}{2}+\gamma_{0}}^{1-\gamma_{0}}\sin (\pi\al) r^{\al}\ma d\al
\]
\[
\geq \sin (\pi (1-\gamma_{0})) r^{1-\gamma_{0}}b_{\mu} \equiv \bar{b}_{\mu} r^{1-\gamma_{0}},
\]
where the constant $\bar{b}_{\mu} > 0$ depends only on $\mu.$ Analogical, for $r>1$
\[
\izj \sin (\pi\al) r^{\al}\ma d\al \geq \bar{b}_{\mu} r^{\frac{1}{2}+\gamma_{0}}.
\]
Then we have
\[
g(t) \leq \frac{1}{\bar{b}_{\mu}\pi} \left(\izj e^{-rt} r^{\gamma_{0}-1}dr +\int_{1}^{\infty} e^{-rt} r^{-(\gamma_{0}+\frac{1}{2})}dr\right)
\]
\[
\leq \frac{1}{\bar{b}_{\mu}\pi} \left(\izj  r^{\gamma_{0}-1}dr +\int_{0}^{\infty} e^{-rt} r^{-(\gamma_{0}+\frac{1}{2})}dr\right)
\]
\[
 =\frac{1}{\bar{b}_{\mu}\pi} \left( \frac{1}{\gamma_{0}} + \Gamma\left(\frac{1}{2}-\gamma_{0}\right)
 t^{\gamma_{0}-\frac{1}{2}}\right).
\]
Due to the fact that $\gamma_{0} \in (0,\frac{1}{4})$ we obtain that  $g \in L^{2}(0,T)$ and so the convolution of $g$ and $\poch \left( k*[u(x,\cdot)- u_{0}(x)](t) \right) $ belongs to $C([0,T];H^{-1}(\Om))$. Finally, from (\ref{uuz}) we have $u(\cdot,0)=u_{0}$ in $H^{-1}(\Om)$, which finishes the proof of theorem~\ref{tw1}.

\label{sectioncont}

\section{Proof of theorem \ref{tw2}. }

We shall show that in the case $L = \Delta$ and under the assumption of theorem~\ref{tw2} we can get additional estimate for approximate sequence and then by weak compactness argument we get more regular solution. Here equation (\ref{nn}) takes the form $\dm u^{n}(x,t) = \Delta u^{n}(x,t) + f^{n}(x,t)$. Multiplying this equality by $\vf_{k}$ and integrating over $\Om$ we obtain that
\[
\iO \dm u^{n}(x,t) \vf_{k}(x)dx = \iO u^{n}(x,t)\Delta \vf_{k}(x)dx + \iO f_{\frac{1}{n}}(x,t)\vf_{k}(x)dx
\]
Let us consider the case when $m=1.$ If $w\in \hk^{3}$, then there exist constants $d_{k}$ such that $w(x)= \sum_{k=1}^{\infty}d_{k}\vf_{k}.$ Denoting $w^{n}(x) = \sum_{k=1}^{n}d_{k}\vf_{k}$ and multiplying the equality above by $d_{k}$ and summing over $k$ from $1$ to $n$ we get
\eqq{
\iO \dm u^{n}(x,t) w(x)dx = \iO u^{n}(x,t)\Delta w(x)dx + \iO f_{\frac{1}{n}}(x,t)w^{n}(x)dx,
}{ms1}
where we skip  superscript $n$ in the first two terms applying  orthogonality condition imposed on $\{ \vk \}_{k\in \mathbb{N}}$.  We introduce the distributed order Riemann-Liouville operator $\rlm$ by the formula $\rlm v = \izj \p^{\beta} v  \cdot \mb d\beta,$
where $\p^{\beta} u$ denotes the Riemann-Liouville fractional derivative (see (\ref{fRL})). We apply the operator $\rlm$ to both sides of (\ref{ms1}). We notice that in this case  $\rlm v = \int_{0}^{\frac{1}{2}} \p^{\beta} v \cdot \mb d\beta$ and
applying proposition 5  \cite{KY} we have  $\p^{\beta} D^{\al}u^{n} = D^{\al+\beta}u^{n}$ for $\al+\beta \leq 1$, thus we obtain
\[
\iO \int_{0}^{\frac{1}{2}}\int_{0}^{\frac{1}{2}} D^{\al+\beta}u^{n}(x,t)\ma \mb d\al d\beta w(x)dx
\]
\eqq{
= \iO\hspace{-0.1cm}\int_{0}^{\frac{1}{2}}\hspace{-0.2cm} \p^{\beta} u^{n}(x,t) \mb d\beta \Delta w(x)dx + \iO\hspace{-0.1cm} \int_{0}^{\frac{1}{2}}\hspace{-0.2cm} \p^{\beta}f_{\frac{1}{n}}(x,t)\mb d\beta w^{n}(x)dx.
}{por}
We first focus on the case $m=1$. Our aim is to estimate the norm of left-hand side of (\ref{por})  in the space $L^{p}(0,T;(\bar{H}^{3})^{\ast})$ for some  $p \in (1,2).$
Comparing the definition of the Caputo derivative and the Riemann-Liouville derivative we obtain the estimate
\[
\abs{\iO \int_{0}^{\frac{1}{2}}\int_{0}^{\frac{1}{2}} D^{\al+\beta}u^{n}(x,t)\ma \mb d\al d\beta w(x)dx}
\]
 \[
 \leq \abs{\iO \dm u^{n}(x,t)\Delta w(x)dx}+ \abs{\iO \int_{0}^{\frac{1}{2}}\hspace{-0.3cm}\frac{t^{-\beta}}{\Gamma(1-\beta)}u^{n}(x,0)\mb d\beta \Delta w(x)dx}
 \]
\[
  + \abs{\iO \int_{0}^{\frac{1}{2}} \p^{\beta}f_{\frac{1}{n}}(x,t)\mb d\beta w^{n}(x)dx}.
\]
We may estimate the second term on the right hand side and we get
\[
\abs{\iO \int_{0}^{\frac{1}{2}}\int_{0}^{\frac{1}{2}} D^{\al+\beta}u^{n}(x,t)\ma \mb d\al d\beta w(x)dx}
\]
\[
\leq \abs{\iO \dm u^{n}\Delta w(x)dx} + 2c_{\mu} \max\{1,t^{-\frac{1}{2}}\}\abs{\iO u^{n}(x,0) \Delta w (x)dx}
\]
\[
 + \abs{\iO \int_{0}^{\frac{1}{2}} \p^{\beta}f_{\frac{1}{n}} (x,t)\mb d\beta w^{n}(x)dx},
\]
where $c_{\mu}$ is defined in (\ref{aab}).  Having in mind that $\Delta w|_{\p \Om} = 0$ we may write
\[
\norm{\int_{0}^{\frac{1}{2}}\int_{0}^{\frac{1}{2}} D^{\al+\beta}u^{n}(\cdot, t)\ma \mb d\al d\beta}_{(\bar{H}^{3})^{\ast}}
\]
\[
= \sup_{\norm{w}_{\bar{H}^{3}}\leq 1} \abs{\iO \int_{0}^{\frac{1}{2}}\int_{0}^{\frac{1}{2}} D^{\al+\beta}u^{n}(x,t)\ma \mb d\al d\beta w(x)dx}
\]
\[
\leq \norm{\dm u^{n}}_{_{H^{-1}(\Om)}} \hspace{-0.9cm} +2c_{\mu}\max\{1,t^{-\frac{1}{2}}\} \norm{\uz}_{_{L^{2}(\Om)}}\hspace{-0.4cm} + \norm{\int_{0}^{\frac{1}{2}}\hspace{-0.2cm} \p^{\beta}f_{\frac{1}{n}} (\cdot, t)\mb d\beta}_{_{(\bar{H}^{3})^{\ast}}}\hspace{-0.7cm}.
\]
We take both sides of last inequality to the power of $p\in (1,2)$ and integrate with respect to~$t.$ Then we obtain that
 \[
\norm{\int_{0}^{\frac{1}{2}}\int_{0}^{\frac{1}{2}} D^{\al+\beta}u^{n} \hspace{0.1cm} \ma \mb d\al d\beta}_{L^{p}(0,T;(\bar{H}^{3})^{\ast})}
\]
\[
\leq c_{1}\hspace{-0.1cm}\left(\norm{\dm u^{n}}_{_{L^{2}(0,T;H^{-1}(\Om))}}\hspace{-0.7cm} +\norm{\uz}_{_{L^{2}(\Om)}}
 \hspace{-0.3cm}+ \norm{\int_{0}^{\frac{1}{2}} \p^{\beta}f_{\frac{1}{n}}\mb d\beta}_{_{L^{2}(0,T;(\bar{H}^{3})^{\ast})}}\right)\hspace{-0.1cm},
\]
where $c_{1} $ depends only on $c_{\mu}$, $p$ and $T$.
We note that
\eqq{
\int_{0}^{\frac{1}{2}} \p^{\beta}f_{\frac{1}{n}} \hj \mb d\beta \longrightarrow \int_{0}^{\frac{1}{2}} \p^{\beta}f \hj \mb d\beta \textrm{ in } L^{2}(0,T;(\bar{H}^{3})^{\ast}),}{ostatnie}
because
\[
\norm{\int_{0}^{\frac{1}{2}} \left(\p^{\beta}f_{\frac{1}{n}} -\p^{\beta}f \right) \hj \mb d\beta}_{_{L^{2}(0,T;(\bar{H}^{3})^{\ast})}} \hspace{-1.1cm} \leq \int_{0}^{\frac{1}{2}}  \hspace{-0.1cm}\norm{\p^{\beta}f_{\frac{1}{n}} -\p^{\beta}f }_{_{L^{2}(0,T;(\bar{H}^{3})^{\ast})}} \hspace{-0.5cm} \mb d\beta
\]
 and from proposition 13 in \cite{KY}  $\p^{\beta}f_{\frac{1}{n}}  \longrightarrow  \p^{\beta}f $ in $L^{2}(0,T;(\bar{H}^{3})^{\ast})$ uniformly with respect to $\beta \in [0,\frac{1}{2}]$. Thus  we obtain (\ref{ostatnie}) and for $n$ large enough we have
\eqq{
 \norm{\int_{0}^{\frac{1}{2}} \p^{\beta}f_{\frac{1}{n}} \hj \mb d\beta}_{_{L^{2}(0,T;(\bar{H}^{3})^{\ast})}}
\leq 2 \norm{\int_{0}^{\frac{1}{2}}  \p^{\beta}f \hj \mb d\beta}_{_{L^{2}(0,T;(\bar{H}^{3})^{\ast})}} \hspace{-0.3cm}.
}{ostatnie2}

Making use of estimate (\ref{dmszac}) and assumptions concerning $\uz$ and $f$ we obtain that the sequence
$\int_{0}^{\frac{1}{2}}\int_{0}^{\frac{1}{2}} D^{\al+\beta}u^{n} \hj \ma \mb d\al d\beta$ is uniformly bounded in $L^{p}(0,T;(\bar{H}^{3})^{\ast})$ for every $p \in (1,2).$ By  weak compactness argument we obtain that for weak solution of (\ref{uklad}) we have
\eqq{
\int_{0}^{\frac{1}{2}}\int_{0}^{\frac{1}{2}} I^{1-(\al+\beta)}[u-\uz]\hj \ma \mb d\al d\beta\in {}_{0}W^{1,p}(0,T;(\bar{H}^{3})^{\ast}),
}{d4szac}
for every $p \in (1,2)$ and by zero we mean that the function vanishes at $t=0$. If we denote
\[
k_{1}(t) = \int_{0}^{\frac{1}{2}}\int_{0}^{\frac{1}{2}} \frac{1}{\Gamma(1-(\al+\beta))} t^{-(\al+\beta)}\ma \mb d\al d\beta,
\]
then (\ref{d4szac}) can be written shorter: $k_{1}\ast [u-u_{0}]\in {}_{0}W^{1,p}(0,T;(\bar{H}^{3})^{\ast})$. Now, we would like to find the operator inverse to  $\ddt k_{1}* \cdot$.  We follow the steps from the second section of the paper. If we investigate the Laplace transform of $k_{1}$, then  we obtain that
\[
\tl{k_{1}}(p) = \int_{0}^{\frac{1}{2}}\int_{0}^{\frac{1}{2}} p^{\al+\beta-1}\ma \mb d\al d\beta.
\]
We note that from assumption (\ref{intmum}) with $m=1$ we get $\overline{\gamma} \in (0, \frac{1}{8})$ such that $\int_{\frac{1}{4}+\overline{\gamma}}^{\frac{1}{2} - \overline{\gamma}} \mu (\al) d \al$ is positive. Thus as in the proof of proposition~\ref{propone} we are able to prove that $F(p)=\frac{1}{p\tl{k_{1}}(p)}$ satisfy assumptions of lemma \ref{odwrotna}. Applying this lemma we obtain that $k_{1}* g_{1}=1$, where $g_{1}$ is given by formula
\[
g_{1}(t)=\frac{1}{\pi}\izi  e^{-rt}\frac{G_{1,s}(r)}{G^{2}_{1,s}(r) + G^{2}_{1,c}(r)}dr,
\]
where
\[
G_{1,s}(r) = \int_{0}^{\frac{1}{2}}\int_{0}^{\frac{1}{2}}\sin \pi (\al+\beta)r^{\al+\beta}\ma\mb d\al d\beta,
\]
\[
G_{1,c}(r) = \int_{0}^{\frac{1}{2}}\int_{0}^{\frac{1}{2}}\cos \pi (\al+\beta)r^{\al+\beta}\ma\mb d\al d\beta.
\]

We will find appropriate estimate for kernel $g_{1}.$ To that end, we may notice that, thanks to the assumption $\int_{\frac{1}{4}}^{\frac{1}{2}}\ma d\al > 0$ we can find $\gamma_{1} \in (0,\frac{1}{8})$ such, that
\eqq{
\int_{\frac{1}{4}+\gamma_{1}}^{\frac{1}{2}-\gamma_{1}}\ma d\al > 0.
}{gamma3}
Using (\ref{gamma3}) we may estimate as a few times before
\[
\int_{0}^{\frac{1}{2}}\int_{0}^{\frac{1}{2}}\sin \pi (\al+\beta)r^{\al+\beta}\ma\mb d\al d\beta
\]
\[
\geq \int_{\frac{1}{4}+\gamma_{1}}^{\frac{1}{2}-\gamma_{1}}
\int_{\frac{1}{4}+\gamma_{1}}^{\frac{1}{2}-\gamma_{1}}\sin \pi (\al+\beta)r^{\al+\beta}\ma\mb d\al d\beta.
\]
For $r \leq 1$ we have
\[
\int_{0}^{\frac{1}{2}}\int_{0}^{\frac{1}{2}}\sin \pi (\al+\beta)r^{\al+\beta}\ma\mb d\al d\beta \geq c r^{1-2\gamma_{1}}
\]
and for $r > 1$
\[
\int_{0}^{\frac{1}{2}}\int_{0}^{\frac{1}{2}}\sin \pi (\al+\beta)r^{\al+\beta}\ma\mb d\al d\beta \geq c r^{\frac{1}{2}+2\gamma_{1}},
\]
where the constant $c> 0$ depends only on $\mu.$
Then we obtain the estimate
\[
g_{1}(t) \leq c_{1} \left[\izj e^{-rt}r^{2\gamma_{1}-1}dr + \int_{1}^{\infty}e^{-rt}r^{-(\frac{1}{2}+2\gamma_{1})}dr\right].
\]
As in the proof of proposition~\ref{oszacg1},  for fixed $T > 0$ we get that $g_{1}(t) \leq c_{2} t^{2\gamma_{1}-\frac{1}{2}}$,
where $c_{2} = c_{2}(\mu, T)$. Finally, we see that
\eqq{
g_{1} \in L^{q}(0,T), \hd  \textrm{ where } \hd q < 2+\frac{8\gamma_{1}}{1-4\gamma_{1}}.
}{g1}
Now  we are ready to prove continuity of solution in the case $m=1.$ From (\ref{d4szac}) we deduce that $k_{1}*[u-u_{0}]$  is absolutely continuous with values in $(\bar{H}^{3})^{\ast}$ and vanishes for $t=0$, thus we have
\eqq{
g_{1} * \poch \left( k_{1}* [u-u_{0}]  \right)(t)= \poch \left( g_{1} *  k_{1}* [u-u_{0}]  \right)(t)= u(\cdot, t) - \uz.
}{uuz1}
If we take $q>2$  satisfying  (\ref{g1}) and set $p=\frac{q}{q-1}$ in (\ref{d4szac}), then  we obtain that   $\poch \left( k_{1}* [u-u_{0}]  \right)$ belongs to  $L^{p}$ and  the convolution on the left-hand side of (\ref{uuz1}) is continuous.   Thus, $u \in C([0,T];(\bar{H}^{3})^{\ast})$ and $u(\cdot,0)=u_{0}$ in $(\bar{H}^{3})^{\ast}$.

Now, we can move on to the general case. To simplify the notation we introduce for natural $k$ and $m$ $I_{k} = [0,\frac{1}{2m}]^{k}$. If  $m>1$ we apply to both sides of equation (\ref{ms1}) the  operators
\[
\int_{I_{k}} \p^{\al_{1}+\cdots +\al_{k}}\prod_{i=1}^{k}\mu(\al_{i})d\al_{i},
\]
where  $k=1,\cdots,m$. Then for each $k=1,\cdots,m$ we obtain for $w \in \bar{H}^{2k+1}$
\[
\iO\int_{I_{k}} \int_{0}^{\frac{1}{2m}}D^{\al+\al_{1}+\cdots +\al_{k}}u^{n}(x,t)\mu(\al)d\al\prod_{i=1}^{k}\mu(\al_{i})d\al_{i}w(x)dx
\]
\[
=\iO\int_{I_{k}} D^{\al_{1}+\cdots +\al_{k}}u^{n}(x,t)\prod_{i=1}^{k}\mu(\al_{i})d\al_{i}\Delta w(x)dx
\]
\[
+\iO\int_{I_{k}}  \frac{t^{-(\al_{1}+\cdots +\al_{k})}}{\Gamma(1-(\al_{1}+\cdots \al_{k}))}u^{n}(x,0)\prod_{i=1}^{k}\mu(\al_{i})d\al_{i}\Delta w(x)dx
\]
\[
 +\iO\int_{I_{k}}  \p^{\al_{1}+\cdots + \al_{k}}f_{\frac{1}{n}}(x,t)\prod_{i=1}^{k}\mu(\al_{i})d\al_{i} w^{n}(x)dx.
\]
Thus, repeating the argument above we get for each $k=1,\cdots,m$
\[
\norm{\int_{I_{k}} \int_{0}^{\frac{1}{2m}} D^{\al+\al_{1}+\cdots +\al_{k}}u^{n}(\cdot ,t)\mu(\al)d\al \prod_{i=1}^{k}\mu(\al_{i})d\al_{i}}_{(\bar{H}^{2k+1})^{\ast}}
\]
\[
\leq \norm{\int_{I_{k}} D^{\al_{1}+\cdots +\al_{k}}u^{n}(\cdot,t)\prod_{i=1}^{k}\mu(\al_{i})d\al_{i}}_{(\bar{H}^{2k-1})^{\ast}} +c t^{-\frac{k}{2m}}\norm{\uz}_{L^{2}(\Om)}
\]
\[
+ \norm{  \int_{I_{k}} \partial^{\al_{1}+\cdots +\al_{k}}\fjn(\cdot,t)\prod_{i=1}^{k}\mu(\al_{i})d\al_{i}}_{(\bar{H}^{2k+1})^{\ast}}.
\]
Summing up these inequalities for  $k=1,\cdots,m$ we obtain that
\[
\norm{\int_{I_{m}} \int_{0}^{\frac{1}{2m}} D^{\al+\al_{1}+\cdots +\al_{m}}u^{n}(\cdot,t)\mu(\al)d\al \prod_{i=1}^{m}\mu(\al_{i})d\al_{i}}_{(\bar{H}^{2m+1})^{\ast}}
\]
\[
\leq \norm{\dm u^{n}(\cdot,t)}_{H^{-1}(\Om)} +c \norm{\uz}_{L^{2}(\Om)}\sum_{k=1}^{m}t^{-\frac{k}{2m}}
\]
\[
+ \sum_{k=1}^{m}\norm{ \int_{I_{k}} \partial^{\al_{1}+\cdots+ \al_{k}}\fjn(\cdot,t)\prod_{i=1}^{k}\mu(\al_{i})d\al_{i}}_{(\bar{H}^{2k+1})^{\ast}}.
\]\
Proceeding as in the proof of (\ref{ostatnie}) and (\ref{ostatnie2}) we get the estimates for $f_{\frac{1}{n}}$ in appropriate spaces. At last, with use of estimate (\ref{dmszac}) we obtain that for  $p \in (1,2)$
\[
\int_{I_{m}} \int_{0}^{\frac{1}{2m}}\hspace{-0.2cm} D^{\al+\sum_{i=1}^{m}\al_{i}}u^{n}\mu(\al)d\al \prod_{i=1}^{m}\mu(\al_{i})d\al_{i} \in L^{p}(0,T;(\bar{H}^{2m+1})^{\ast}).
\]
By weak compactness argument we have
\eqq{
k_{m}* [u-u_{0}] \in {}_{0}W^{1,p}(0,T;(\bar{H}^{2m+1})^{\ast}),
}{xxa}
where
\[
k_{m}(t)= \hspace{-0.1cm}\int_{I_{m}}\int_{0}^{\frac{1}{2m}}\hspace{-0.2cm}  \frac{1}{\Gamma(1-\beta)}t^{-\beta} \mu(\al) d\al \prod_{k=1}^{m} \mu(\al_{k})d\al_{k},
\]
where $\beta= \al + \sum_{i=1}^{m} \al_{i}$. From the assumption (\ref{intmum}) we deduce that
\[
\exists \hj \gamma_{m} \in  \left( 0, \frac{1}{4m(m+1)}  \right) \hd \hd \hd \int_{\frac{1}{2(m+1)}+\gamma_{m}}^{\frac{1}{2m}-\gamma_{m}} \ma d\al >0 .
\]
Then using lemma~\ref{odwrotna} we obtain the operator inverse to $\ddt  k_{m}* \cdot $, which is defined as a convolution with function $g_{m}$ given by the formula
\[
g_{m}(t) =  \frac{1}{\pi} \int_{0}^{\infty} e^{-rt}\frac{G_{m,s}(r)}{G^{2}_{m,s}(r) + G^{2}_{m,c}(r) }dr,
\]
where
\[
G_{m,s}(r) = \int_{0}^{\frac{1}{2m}} \int_{I_{m}} \sin(\beta \pi) r^{\beta} \prod_{i=1}^{m} \mu(\al_{i}) d\al_{i} \ma d \al,
\]
\[
G_{m,c}(r) = \int_{0}^{\frac{1}{2m}} \int_{I_{m}} \cos(\beta \pi) r^{\beta} \prod_{i=1}^{m} \mu(\al_{i}) d\al_{i} \ma d \al.
\]
Further we obtain the estimate
\[
g_{m}(t)\leq c \left[ \izj\hspace{-0.2cm} e^{-rt } r^{-(\frac{1}{2m}-\gamma_{m})(m+1)} dr +\hspace{-0.2cm} \int_{1}^{\infty}\hspace{-0.2cm} e^{-rt } r^{-(\frac{1}{2(m+1)}+\gamma_{m})(m+1)} dr \right].
\]
Thus we deduce that
\[
g_{m}(t) \leq c(\mu, T) t^{(\frac{1}{2(m+1)}+\gamma_{m})(m+1)-1}
\]
and
\eqq{
g_{m} \in L^{q}(0,T)  \m{ for } q<2+\frac{4\gamma_{m}(m+1)}{1-4\gamma_{m}}.
}{xxb}
By (\ref{xxa}) we deduce that function $k_{m}*[u-u_{0}]$ is absolutely continuous with values in $(\hk^{2m+1})^{*}$ and vanishes at $t=0$, thus we have
\eqq{
g_{m} * \poch \left( k_{m}* [u-u_{0}]  \right)(t)= \poch \left( g_{m} *  k_{m}* [u-u_{0}]  \right)(t)= u(\cdot,t) - \uz.
 }{uuzmx}
If we take $q>2$  satisfying  (\ref{xxb}) and set $p=\frac{q}{q-1}$ in (\ref{xxa}), then  we obtain that   $\poch \left( k_{1}* [u-u_{0}]  \right)$ belongs to  $L^{p}$ and  the convolution on the left-hand side of (\ref{uuzmx}) is continuous. Thus  $u\in C([0,T];(\hk^{2m+1})^{*})$ and $u(\cdot,t)=u_{0}$ in $(\hk^{2m+1})^{*}$. Therefore the proof of theorem~\ref{tw2} is finished.

\section{Proof of theorem~\ref{regularne}.}

In this section we prove the existence of regular solution. We shall show additional estimate for the sequence of approximate solution $u^{n}$ given by remark~\ref{remone}. We multiply (\ref{przyb}) by $\lambda_{m} c_{n,m}(t)$ and sum over $m=1, \dots, n $. Then we have
\[
-\hspace{-0.1cm} \io \hspace{-0.1cm} \dm \un (x,t ) \cdot \lap \un (x,t) dx =\hspace{-0.2cm} \sum_{i,j=1}^{N}\hspace{-0.1cm}\iO \hspace{-0.1cm} a_{i,j}^{n}(x,t)D_{j}u^{n}(x,t)D_{i}\lap u^{n}(x,t)dx
\]
\[
 -\sum_{j=1}^{N}\iO \hspace{-0.1cm} b_{j}^{n}(x,t)D_{j} u^{n}(x,t)\cdot \lap u^{n}(x,t)dx - \iO \hspace{-0.1cm} c^{n}(x,t)u^{n}(x,t)\cdot \lap \un (x,t)dx
 \]
 \[
 - \io  f_{\frac{1}{n}}(x,t) \cdot \lap u^{n}(x,t)dx.
\]
If we integrate by parts and apply proposition~9 \cite{KY},  then we get
\[
 \io \hspace{-0.2cm}\dm \nabla \un (x,t ) \cdot \nabla \un (x,t) dx + \frac{\lambda}{16}\hspace{-0.1cm} \io \hspace{-0.2cm} |D^{2}\un (x,t)|^{2} dx \leq c\hspace{-0.1cm}\io \hspace{-0.2cm} |\nabla \un(x,t) |^{2}dx
\]
\[
+ \frac{4}{\lambda} \| f_{\frac{1}{n}}(\cdot ,t) \|_{\ld}^{2}+\frac{4}{\lambda} \|  b^{n}(\cdot, t) \nabla \un (\cdot, t )  \|_{\ld}^{2} + \frac{4}{\lambda} \| c^{n}(\cdot, t)  \un (\cdot, t )  \|_{\ld}^{2}
\]
Estimating the last two terms in the inequality above and applying  lemma~\ref{Alemat}  with $w= \nabla \un $ leads to
\[
\dm \norm{ \nabla u^{n}(\cdot,t)}_{L^{2}(\Om)}^{2} + \frac{\lambda}{32} \norm{D^{2} u^{n}(\cdot,t)}^{2}_{L^{2}(\Om)}
\]
\[
\leq \bar{c}\io |\nabla \un(x,t) |^{2}dx + \frac{8}{\lambda} \| f_{\frac{1}{n}}(\cdot ,t) \|_{\ld}^{2},
\]
where $\bar{c}$ depends only on  $\max_{i,j}\|
\nabla \aij  \|_{ L^{\infty}(\Omega^{T})}$, the regularity of
$\partial \Omega$, $p_{1}$, $p_{2}$, $\lambda$ and norms $ \| b \|_{L^{\infty}
(0,T; L^{\frac{2p_{1}}{p_{1}-2}}(\Omega))}$,
$\| c \|_{L^{\infty}(0,T; L^{\frac{p_{2}}{p_{2}-2}}\cap L^{2}(\Omega))}$. If we integrate the above inequality over $(0,t)$,  then we obtain
\[
\izj \ija \| \nabla \un (\cdot, t) \|_{\ld}^{2} \ma d\al  + \frac{\lambda}{32} \izt \norm{D^{2} u^{n}(\cdot,\tau)}^{2}_{L^{2}(\Om)}  \dt
\]
\[
\leq \bar{c}\izt  \| \nabla \un(x,\tau ) \|_{\ld}^{2}\dt  + \frac{8}{\lambda} \izt \| f_{\frac{1}{n}}(\cdot ,\tau) \|_{\ld}^{2} \dt
\]
\[
 + \izj \ija \| \nabla \un (\cdot, 0) \|_{\ld}^{2} \ma d\al
\]
\[
\leq \bar{c}_{0}\left(  \izt \| f(\cdot ,\tau) \|_{\ld}^{2} \dt + \|  u_{0} \|_{H^{1}(\Omega)}^{2} \right)+\delta_{n},
\]
where in the last inequality we used lemma~\ref{oszac1} and $\bar{c}_{0}$ depends only on $\bar{c}$, $\mu$ and $T$.  From this inequality we may deduce that the norm of $\un$ in $L^{2}(0,T;H^{2}(\Omega))$ is uniformly bounded and repeating the procedure from section~\ref{limpass} we obtain a uniform bound for the norm of $\izj \ija [\un - \un_{0}] \hj \ma d \al $ in ${}_{0}H^{1}(0,T;\ld)$. Applying the weak compactness argument we finish the proof of  the first part of theorem~\ref{regularne}.

\no To get the continuity of solution at $t=0$ we use the argument from section~\ref{sectioncont} and immediately we get that the convolution of $g$ and  $\poch \left( k*[u(x,\cdot)- u_{0}(x)](t) \right)$ belongs to  $  C([0,T];\ld)$, hence  from (\ref{uuz}) we have $u(\cdot,0)=u_{0}$.

\section{Appendix}

We recall here lemma 2.1 from \cite{Laplace}. We also give a proof because in our opinion assumption $4$ is needed.

\begin{lem}\label{odwrotna}
Let $F$ be a complex function, satisfying following assumptions:
\begin{enumerate}[1)]
\item
$F(p)$ is analitic in $\mathbb{C} \setminus (-\infty, 0].$
\item
The limit $F^{\pm}(t):=\lim\limits_{\vf\rightarrow \pi^{-}}F(te^{\pm i \vf})$ exists for a.a. $t>0$ and $F^{+} = \overline{F^{-}}$.
\item
For each   $0<\eta<\pi$
\begin{enumerate}[a)]
\item
 $|F(p)| = o(1)$, as $|p|\rightarrow \infty$ uniformly on $ |\arg(p)| < \pi - \eta$
\item
$|F(p)|= o(\frac{1}{|p|})$, as $|p|\rightarrow 0$ uniformly on $|\arg(p)|< \pi - \eta$.
\end{enumerate}
\item
There exists $\ve_{0} \in (0,\frac{\pi}{2})$ and a function $a=a(r)$ such that \hd $\forall  \vf \in (\pi - \ve_{0}, \pi)$  the estimate  $\abs{F(re^{\pm i\vf})}   \leq a(r)$ holds, where
\[
 \frac{a(r)}{1+r} \in L^{1}(\mr_{+}).
\]
\end{enumerate}
Then for $p\in \mathbb{C}$ such that $\Re{p}>0$ we have
\[
F(p) = \izi e^{-xp}f(x)dx, \textrm{  where  }f(x) = \frac{1}{\pi}\izi e^{-rx} \Im(F^{-}(r))dr.
\]
\end{lem}

\begin{proof}
We follow \cite{Laplace}: we fix $p\in \mathbb{C}$ such that $\Re{p}>0$. Then we choose $r_{1}$, $r_{2}$ and $\phi$ such that $0<r_{1}<|p|<r_{2}$ and $\phi \in (\pi - \ep_{0}, \pi)$. We denote by $C^{\phi}_{r_{1},r_{2}}$ a positively oriented closed contour created by two segments $\{z\in \mathbb{C}: r_{1}\leq |z| \leq r_{2}, \hd \arg(z)= \pm \phi  \}$ and two arcs $\{z\in \mathbb{C}: \hd |z|= r_{i}, \hd |\arg{z}|<\phi \}$, $i=1,2$. Then from Cauchy formula we have
\[
F(p)= \frac{1}{2\pi i } \int_{C^{\phi}_{r_{1},r_{2}}} \frac{F(z)}{z-p}dz .
\]
The assumption $3$ allows us to  take the limit $r_{1}\rightarrow 0$ and $r_{2} \rightarrow \infty$ and we get
\[
F(p)= \frac{1}{2\pi i } \izi \left[   \frac{F(re^{i\phi} )e^{i \phi}}{p - re^{i\phi}} -  \frac{F(re^{-i\phi} )e^{-i \phi}}{p - re^{-i\phi}}\right] dr ,
\]
and by assumption $4$ this integral is absolutely convergent. Because $\Re{re^{\pm i \phi}}<0 $ we can write
\[
\frac{1}{p - re^{\pm i \phi }} = \izi e^{-x(p-re^{\pm i \phi})} dx
\]
and we have $F(p)=$
\[
\frac{1}{2\pi i }\hspace{-0.1cm} \izi \hspace{-0.1cm}\left[   F(re^{i\phi} )e^{i \phi}\hspace{-0.2cm}\izi\hspace{-0.2cm} e^{-x(p-re^{ i \phi})} dx\hspace{-0.1cm} - \hspace{-0.1cm} F(re^{-i\phi} )e^{-i \phi}\hspace{-0.2cm}\izi \hspace{-0.2cm} e^{-x(p-re^{- i \phi})} dx\right]\hspace{-0.1cm} dr .
\]
The above integral is absolutely convergent because we have
\[
\left| e^{\pm i \phi} F(re^{\pm i \phi  }) e^{-xp + xre^{\pm i \phi }} \right| \leq a(r) e^{x(-\Re p+r\cos{\phi})}
\]
\[
\leq a(r) e^{x[-\Re p+r\cos(\pi- \ep_{0})]}
\]
and by assumption $4$ we get
\eqq{
\izi \hspace{-0.1cm} \hspace{-0.1cm} \izi \hspace{-0.1cm} \hspace{-0.1cm} \hspace{-0.1cm}  a(r) e^{x[-\Re p+r\cos(\pi- \ep_{0})]} dxdr = \hspace{-0.1cm} \izi \hspace{-0.1cm}
\hspace{-0.1cm} \hspace{-0.1cm}  \frac{a(r)}{\Re p - r\cos(\pi - \ep_{0})} dr <\hspace{-0.1cm} \infty.
}{nowec}
Therefore we may apply Fubini theorem and we have
\eqq{
F(p)= \izi e^{-xp} f(x,\phi) dx,
}{noweb}
where
\eqq{
f(x, \phi )= \frac{1}{2\pi i } \izi \left[  e^{i \phi + rxe^{i\phi }} F(re^{i \phi}) -  e^{-i \phi + rxe^{-i\phi }} F(re^{-i \phi})  \right]dr
}{nowea}
and by assumption $4$ for each $x\in (0,\infty)$ and $\phi \in (\pi- \ep_{0}, \pi)$ the integral (\ref{nowea}) is absolutely convergent. We shall take the limit $\phi \rightarrow \pi$ in  (\ref{noweb}). First we note that
\[
|e^{- xp} f(x,\phi)| \leq e^{- x \Re p} \izi e^{rx \cos{\phi}} \left( |F(re^{i\phi})|+|F(re^{-i\phi})| \right)dr
\]
\[
\leq 2 e^{-x \Re p} \izi e^{rx \cos(\pi - \ep_{0})} a(r) dr
\]
is in $L^{1}(\mathbb{R}_{+})$, because (\ref{nowec}) holds. Hence by Lebesgue  dominated convergence we have
\[
F(p)= \izi e^{-xp} \lim_{\phi \rightarrow \pi} f(x,\phi) dx.
\]
Applying again  assumption $4$ together with  Lebesgue  dominated convergence theorem  we get
\[
F(p)= \izi e^{-xp}  f(x,\pi) dx, \hd \m{ and } \hd f(x, \pi )= \frac{1}{\pi} \izi e^{- rx } \Im(F^{-}(r))dr.
\]

\end{proof}

\end{document}